\documentclass[12pt]{article}
\usepackage{graphicx}
\usepackage{amsmath}
\usepackage{amsthm}
\usepackage{graphics}
\usepackage{epsfig}
\usepackage{amssymb}
\usepackage{enumitem}
\usepackage{algorithm}
\usepackage{algpseudocode}
\usepackage{hyperref}

\usepackage[toc,page]{appendix}
\usepackage{xcolor}
 \usepackage{setspace}
\usepackage{extarrows}

\RequirePackage[numbers,sort]{natbib}
\usepackage{subfig}
\usepackage{lmodern}

\usepackage{pdflscape}
\usepackage{subfig}
\usepackage{enumitem}
\allowdisplaybreaks

\newcommand{\bfS}{\mathbf{S}}
\newcommand{\nt}{\lfloor nt \rfloor}
\newcommand{\ns}{\lfloor ns \rfloor}

\newcommand{\R}{\mathbb{R}}
\newcommand{\lb}{\left(}
\newcommand{\rb}{\right)}
\newcommand{\Var}{\operatorname{Var}}
\newcommand{\cov}{\operatorname{cov}}
\newcommand{\E}{\mathbb{E}}
\newcommand{\N}{\mathbb{N}}

\newcommand{\bfR}{\mathbf{R}}

\newcommand{\bfSigma}{\mathbf{\Sigma}}
\newcommand{\bfX}{\mathbf{X}}
\newcommand{\bfI}{\mathbf{I}}

\newcommand{\bfK}{\mathbf{K}}
\newcommand{\MP}{Mar\v cenko--Pastur }

\newcommand{\cond}{\stackrel{\mathcal{D}}{\to}}
\newcommand{\conp}{\stackrel{\mathbb{P}}{\to}}
\newcommand{\tr}{\operatorname{tr}}
\newcommand{\D}{\mathbf{D}}

\newcommand{\T}{\mathbf{I}} 

\newcommand{\inv}{^{-1}}
\newcommand{\sq}{^{\frac{1}{2}}}

\newcommand{\PR}{\mathbb{P}}

\newcommand{\bfx}{\mathbf{x}}
\newcommand{\bfy}{\mathbf{y}}
\newcommand{\op}{o_{\PR}(1)}

\newcommand{\bfA}{\mathbf{A}}
\newcommand{\bfB}{\mathbf{B}}
\newcommand{\bfC}{\mathbf{C}}
\newcommand{\su}{\underline{m}}

\newtheorem{theorem}{Theorem}[section]
\newtheorem{lemma}{Lemma}[section]

\newtheorem{remark}{Remark}[section]

\numberwithin{equation}{section}

\allowdisplaybreaks

\begin{document}
\title{Detecting Spectral Breaks in Spiked Covariance Models}
\date{\today}
\author{Nina Dörnemann\thanks{University of California, Davis. E-Mail: ndoernemann@ucdavis.edu} \and Debashis Paul\thanks{University of California, Davis and Indian Statistical Institute, Kolkata. E-Mail: debpaul@ucdavis.edu}}
\maketitle

\begin{abstract}
In this paper, the key objects of interest are the sequential covariance matrices $\mathbf{S}_{n,t}$ and their largest eigenvalues. Here, the matrix $\bfS_{n,t}$ is computed as the empirical covariance associated with observations $\{\bfx_1,\ldots,\bfx_{ \lfloor nt \rfloor } \}$,  
for $t\in [0,1]$. The observations $\bfx_1,\ldots,\bfx_n$ are assumed to be i.i.d. $p$-dimensional vectors with zero mean,
and a covariance matrix that is a fixed-rank perturbation of the identity matrix.
Treating $\{
\mathbf{S}_{n,t}\}_{t \in [0,1]}$ as a matrix-valued stochastic process indexed by $t$, we study 
the behavior of the largest eigenvalues of $\mathbf{S}_{n,t}$, as 
$t$ varies, with $n$ and $p$ increasing simultaneously, so 
that $p/n \to y \in (0,1)$. As a key contribution of this work, we establish the weak convergence of the stochastic process corresponding to the sample spiked eigenvalues, if their population counterparts exceed the critical phase-transition threshold. Our analysis of the limiting process is fully comprehensive revealing, in general, non-Gaussian limiting processes.

As an application, we consider a class of change-point problems, where the interest is in detecting structural breaks in the covariance caused by a change in magnitude of the spiked eigenvalues. For this purpose, we propose two different maximal statistics corresponding to centered spiked eigenvalues
of the sequential covariances. We show the existence of limiting null distributions for these statistics, and prove consistency 
of the test under fixed alternatives. Moreover, we compare the behavior of the proposed tests through a simulation study.
\end{abstract}

MSC2020 subject classifications: Primary 15A18, 60F17; secondary 62H15.


\section{Introduction}

Principal component analysis (PCA) has found widespread applications in social, behavioral and biological sciences and in different branches of engineering \citep{feng2000human, fabozzi2007robust, ruppert2011statistics,lorenz1956empirical} due to its appeal in representing the 
structural aspects of the signal in a suitably low-dimensional linear space. Research on high-dimensional PCA over the last three decades (cf. 
\cite{johnstonePaul2018}) has brought to focus the various statistical phenomena associated with large dimensionality, such as non-trivial eigenvalue bias, 
including a phase transition phenomenon, as well as 
lack of consistency of the corresponding sample eigenvectors, as estimators of their population counterparts. 
For theoretical analyses
in the context of high-dimensional PCA, the so-called ``spiked covariance'' model is widely used,
whereby
a population covariance is represented as a sum of a 
low dimensional ``signal covariance'' and an isotropic 
noise covariance. In other words, the population covariance is a
low-rank perturbation of a positive multiple of the identity matrix. Following this description, the ``signal covariance'' is supposed to represent the components 
of variability that are driven by some underlying correlated
structure, whereas all the coordinates of the data 
matrix are contaminated by additive, isotropic noise.
From this perspective, the problem of reducing the 
dimensionality of the data boils down to identifying
the principal directions associated with the ``spiked eigenvalues'', i.e., the eigenvalues whose magnitudes 
are bigger than the noise variance, and then projecting the data to this ``signal-bearing'' subspace.

\subsection{Related literature on RMT and change-point problems}

Over the last two decades, there have been considerable developments in understanding the properties 
of PCA for high-dimensional data. 
A significant part of these studies adopted the so-called random
matrix theory (RMT) regime of asymptotics (\cite{bai2004, paulaue2014,yao2015, tao2023topics}), whereby the dimension
of the data is assumed to grow proportionally to the sample
size. There are statistical implications of this asymptotic 
paradigm that are fundamentally different from the classical 
paradigm of fixed dimensional large sample asympotics. 
One of the most noteworthy findings in this context is the
phase transition phenomena for the leading sample eigenvalues and eigenvectors \cite{bbp, bao2022statistical,  baik2006eigenvalues}.
In particular, the empirical eigenvalues exhibit distinct fluctuations depending on the relative size of their population counterparts to the bulk of the spectrum: in the subcritical regime, Tracy-Widom fluctuations of order $n^{-2/3}$ are observed, while in the supercritical regime, Gaussian fluctuations of order $n^{-1/2}$ are prevalent (see, e.g., \cite{elkaroui2007, onatski_tw, bao2015, leeschnelli2016, ding2018necessary, knowles2017anisotropic, schnelli2023convergence} on the subcritical case and \citep{paul2007asymptotics, zhang2022asymptotic, bai2008central, BAI2012167, cai_han_pan, jiang_bai, jiang_bai_2} for the supercritical case, among many others).
Especially, under the ``spiked covariance'' framework introduced in \cite{johnstone2001distribution}, 
there have been extensive analyses on the behavior of the 
leading eigenvalues and eigenvectors of the sample covariance 
matrix under the spiked covariance model within the RMT framework (see e.g.
\cite{johnstonePaul2018} and  the references therein).

As the understanding of high-dimensional PCA continues to evolve, it is natural to utilize the associated tools and ideas to the exploration of change-point problems. 
Problems of detecting structural breaks in high-dimensional data are increasingly ubiquitous in statistics due to increasing availability of complex data types. Here, a change-point problem consists of finding stationary segments in a piece-wise stationary time series. There is an extensive literature on change-point problems involving fixed dimensional observations recorded in time. Much of the existing literature has focused on structural changes through shifts in the mean (see \cite{jirak2015,dettegoesmann2020,chengwangwu,wangvolgushev,zhangetal2022} among many others). However, in recent times, there is increasing focus on change-point problems where the changes occur in higher order moments, especially the covariance structure. We refer to \cite{ ChenGupta2004, galeanopena2007, aueEtAl2009, dettewied2016} for some early references on the low-dimensional case. At the same time, there has been an explosion in literature on high-dimensional inference, including in the context of change point detection. Nevertheless, the problem of detection of structural breaks in the covariance of high-dimensional observations has received relatively limited attention so far. 

For fixed dimension, \cite{aueEtAl2009}, and \cite{aueHorvath2013} developed a nonparametric test based on the CUMSUM approach by means of aggregating the coordinates of the centered sequential covariance matrices.  In the work \cite{kao}, the authors considered the case of increasing $p$ and proposed a change point analysis scheme that relies on PCA. \cite{wangYuRinaldo2021} proposed a scheme for multiple change point detection for covariances when $p=O(n/\log n)$, by utilizing a binary segmentation approach applied to the operator norm of a CUMSUM statistics for the covariances. \cite{dette2022estimating} considered the problem of detecting change points in the covariance matrices of a sequence of observations, when the dimension is substantially larger than the sample size.  Their approach consists of first finding an appropriate estimate of a suitable discrepancy measure between individual coordinates of the covariance matrices before and after a hypothetical change point, followed by a constructing a weighted sum of the aforementioned statistic. Cutoff values for this test statistic are determined by a bootstrap procedure. 
Another change-point test in the RMT-framwork was proposed by \cite{ryankillick2023} comparing sequentially multivariate ratios of two covariance matrices corresponding to data collected before and after a potential change point. The methodology relies on a point-wise central limit theorem for F-Matrices in RMT given in \cite{zheng2012} combined with a Bonferroni correction. Consequently, their test may exhibit a conservative performance in applications, as pointed out in Section 4 of \cite{ryankillick2023}. Such drawbacks can be overcome by studying the sequential sample covariance model coupled with the theory of stochastic processes. Initial progress in this endeavor has been made in the work \cite{dornemann2021linear}. In the aforementioned paper, the authors studied the behavior of the renormalized linear spectral statistics corresponding to the sequential covariances and proved the weak convergence of this stochastic process to a Gaussian process. Based on the fluctuations of the limiting process, a change-point test of the sphericity assumption of covariance matrices was developed.

However, none of the works mentioned here specifically addresses the problem of detecting changes in the covariance when the change is only through changes in the leading eigenvalues of the covariance. Indeed, monitoring changes in a large-dimensional covariance structure induced by a low-rank component is motivated by relevant statistical applications, such as array processing (see \cite{beisson2024new} and references therein). 
Nevertheless, in a closely related work, \cite{horvathRice2019} studied the behavior of the process of fluctuations of the eigenvalues of centered sequential sample covariance matrices based on a vector valued process having a factor-model structure with stationary time-dependence. However, their framework assumes that the sample size grows much faster than the dimension, in particular, $p = o(\sqrt{n})$. Under this framework, and assuming that the leading eigenvalues of the population covariance are well-separated, they derive a Gaussian process limit for the fluctuations of the eigenvalues largest eigenvalue, around its population counterpart, of the sequential covariance process. They apply the results to devise a test for detection of a change point in the factor loadings, 
assuming that the change occurs simultaneously across all 
coordinates. Notice that, since dimension is much smaller than the sample size, the eigenvalues of sequential covariances do not entail a bias, which considerably simplifies the description of the limiting form of the process of eigenvalue fluctuations.

\subsection{An overview of our results}

In the present paper, we 
focus on a type of change-point problem where the 
structural change in the covariance is only through 
a change in the relative magnitude of the signal (or 
spiked) eigenvalues, but not in the directionality 
of the signal. 
For this purpose, we work under the
setup of a spiked covariance model, where $n,p \to \infty$ such that the
ratio $p/n$ converges to a nonzero constant, and the number of spiked
eigenvalues is fixed. The main contributions of this work can be summarized as follows.

\begin{enumerate}
    \item \textit{Analysis of the process of sample spiked eigenvalues}
 \\  
In the assumed RMT framework, we first consider
the limits of the leading sample eigenvalues for the sequential
covariance matrices, corresponding to any fixed fraction of
the data, assuming that the population spikes are above the 
corresponding phase-transition threshold. This implies that the 
sample spiked eigenvalues fluctuate around these asymptotic 
limits that are upwardly biased compared to the corresponding 
population spiked eigenvalues. Then we characterize these 
fluctuations, by establishing a limit of the corresponding 
centered process, as the fraction of data varies over the interval 
$[t_0,1]$ for a suitable $t_0>0$. In fact, we prove the weak convergence in an infinite-dimensional function space by the means of investigating the finite-dimensional distributions and establishing asymptotic tightness. Notably, our result incorporates the joint weak convergence of all spiked processes, and thus specifies not only the marginal but the joint limiting process of the signal eigenvalues. Also, our analysis of the limiting  process is fully comprehensive including an explicit representation of the covariance structure.

\item \textit{Change-point problem in high-dimensional row-rank models}
\\
After establishing the convergence of the 
centered and scaled stochastic process of the spiked eigenvalues in the sequential covariance model, we turn to
a change-point problem involving the structure of the underlying
population covariance matrices. Specifically, we aim to detect a potential change in the covariances within the observed time domain, while the structural break occurs only in terms of the magnitude of the spiked eigenvalues of the covariance. Following the commonly used interpretation in factor models, we may treat the principal components associated with the spiked eigenvalues as the signal-bearing coordinates. Then, the particular hypothesis emphasizes that the change may 
only affect the strengths of ``signal'' in the data, but 
not their direction. Under such 
a setting, the changes may be captured by the deviation of 
the sample spiked eigenvalues associated with the sequential 
covariance matrices from the ``null case'' of no change point. More precisely, we test the null hypothesis of all covariances matrices matching a prescribed matrix against the alternative of deviations from this given structure. For this purpose, we propose two different maximal statistics involving the centered spiked eigenvalues of sequential sample covariances. We show that the proposed tests have asymptotically controlled significance levels, and provide consistency results under fixed alternatives.

\item \textit{Accounting for the case of unknown signal strength}
\\
In many applications, the spiked eigenvalues of the population covariance matrices under the null case act as nuisance parameters, and the user is interested in detecting any structural breaks in this low-rank components no matter what their actual values are. Such questions pose the challenge of  estimating the spikes from observed data. 
As a step towards resolving this nontrivial problem, we assume that we have a preliminary, independent data set with the same covariance structure as in the null case. We use this initial data to estimate the 
spiked eigenvalues and then use a plug-in version of the 
maximal statistics that capture the fluctuations of the 
renormalized statistics from the baseline. The plug-in approach
leads to a modification in the limiting process which we quantify, and use the latter process to formulate our test 
for the existence of a change-point. We demonstrate the performance of the proposed test through numerical simulations.
\end{enumerate}

\subsection*{Organization} 

The rest of the paper is organized as follows.
In Section \ref{sec_spike_model}, we introduce the spiked covariance model and define the sequential sample covariance process.
In Section \ref{sec_main_results_eigenval}, we state the 
main results regarding the first order limits and weak convergence of spiked
eigenvalues of the sequential sample covariance process, after appropriate renormalization.
In Section \ref{sec_test}, we introduce the change-point
hypotheses and develop tests based on spiked eigenvalues 
of the sequential sample covariances. Moreover, we investigate the finite-sample performance of the proposed tests. 
Proofs of the main
results are given in Section \ref{sec_proof_main}, which has also an outline about the main steps in Section \ref{sec_proof_outline}. The proofs
of some auxilliary results, and a some mathematical details 
are placed in the Appendices.

\section{The spiked covariance model} \label{sec_spike_model}

Let $\bfx=(\xi^\top, \eta^\top)^\top$ be a centered random vector, where $\xi=(\xi(1), \ldots, \xi(M))^\top \in \R^M$ and $\eta=(\eta(1), \ldots, \eta(p))^\top \in \R^p$ are independent. We assume that the components of $\bfx$ admit finite moments of order $4$ and coordinates of $\eta$ are i.i.d. with variance equal to $1$. Moreover, suppose that the covariance matrix of $\xi$ is denoted by $\operatorname{cov}(\xi)=\bfSigma.$ Then, the covariance matrix of the generic element $\bfx$ can be decomposed as
\begin{align*}
    \cov(x) = \begin{pmatrix}
 \bfSigma & \mathbf{0} \\ 
 \mathbf{0} & \bfI_p
    \end{pmatrix}.
\end{align*}
Here, $\mathbf{0}$ denotes a matrix of appropriate dimension filled with zeros and $\bfI_p$ denotes the $p-$dimensional identity matrix. Note that we sometimes omit the index and just write $\bfI$ if the dimension is clear from the context. 
To introduce the spiked covariance model, we follow the framework of \cite{bai2008central, qinwen2014joint} and assume that $\bfSigma \in \R^{M\times M}$ admits $K$ distinct eigenvalues $\alpha_1, \ldots, \alpha_K \notin \{0,1\}$ with multiplicities $n_1, \ldots, n_K$ such that $n_1+\ldots + n_K=M$. (Note that the matrix $\bfSigma$ is fixed and does not depend on $n.$) Therefore, $\cov(\bfx)$ has eigenvalues $1, \alpha_1, \ldots, \alpha_K$, and $\alpha_1, \ldots, \alpha_K$ are called the spiked eigenvalues. More precisely, we choose $1 \leq M_a, M_b \leq K$ such that exactly $M_b$ of the spiked eigenvalues $\alpha_1, \ldots, \alpha_K$ are larger than $ 1+ \sqrt{y_{(t_0)}}$ and exactly $M_a$ are smaller than $1 - \sqrt{y_{(t_0)}}$ for some $t_0 \in (0,1)$. Thus, we have 
\begin{align*}
    \alpha_1 > \ldots > \alpha_{M_b} > 1 + \sqrt{y_{(t_0)}}, \quad 
    \alpha_M < \ldots < \alpha_{M - M_a + 1} < 1 - \sqrt{y_{(t_0)}} 
\end{align*}
For $k \in \{ 1, \ldots, M_b\} \cup \{ M - M_a + 1, \ldots, M\} $, let $J_k$ denotes the set of indices corresponding to the eigenvalue $\alpha_k$ with multiplicity $n_k.$ Formally, we let $s_i = n_1 + \ldots + n_i $ for $1 \leq i \leq M_b$ and $t_j= n_M + \ldots + n_j $ for $M - M_a + 1 \leq j \leq M.$ Then, we may write
\begin{align*}
    J_k = 
    \begin{cases}
        \{ s_{k-1} +1, \ldots , s_k \} & : \alpha_k > 1 + \sqrt{y_{(t_0)}}, \\
        \{ t_k, \ldots, t_{k-1}+1\} & : \alpha_k < 1 - \sqrt{y_{(t_0)}}.
    \end{cases}
\end{align*}
Let $\bfx_i=(\xi_i^\top, \eta_i^\top),~ 1\leq i \leq n,$ be a sample of independent random vectors distributed according to $\bfx$. Then, we define the sequential sample covariance matrix
\begin{align*}
    \bfS_{n,t} = \frac{1}{n} \sum_{i=1}^{\nt} \bfx_i \bfx_i^\top 
    = \begin{pmatrix}
        \bfS_{1,t} & \bfS_{12,t} \\
        \bfS_{21,t} & \bfS_{2,t} 
        \end{pmatrix}
        = \begin{pmatrix}
        \bfX_{1,t} \bfX_{1,t}^\top & \bfX_{1,t} \bfX_{2,t}^\top  \\
        \bfX_{2,t} \bfX_{1,t}^\top & \bfX_{2,t} \bfX_{2,t}^\top
        \end{pmatrix}, 
        \quad t\in [0,1],
\end{align*}
where
\begin{align*}
    \bfX_{1,t} = \frac{1}{\sqrt{n}} \lb \xi_1, \ldots, \xi_{\nt} \rb  \in \R^{M\times \nt} , ~ 
    \bfX_{2,t} = \frac{1}{\sqrt{n}} \lb \eta_1, \ldots, \eta_{\nt} \rb \in \R^{p\times \nt} .  
\end{align*}
In general, we use the notation $\lambda_1(\bfA) \geq \ldots \geq \lambda_p(\bfA)$ for the ordered eigenvalues of a symmetric $p\times p$ matrix $\bfA.$
More specifically, $\lambda_{n,1,t} = \lambda_1(\bfS_{n,t}) \geq \ldots \geq \lambda_p(\bfS_{n,t}) = \lambda_{n,p,t}$ denote the ordered eigenvalues of $\bfS_{n,t}$, which depend on the sequential parameter $t$ in a highly non-trivial way.  
For the asymptotic results, we assume that the number $M$ of spikes is fixed, and the dimension $p=p_n$ is of the same order as the sample size such that $\lim_{n\to\infty} p/n = y \in (0,1).$
It is known (see, e.g., \cite{baik2006eigenvalues}), that for $j \in \{ 1, \ldots, M_b\} \cup \{ M - M_a + 1, \ldots, M\} $
\begin{align}
    \lambda_{n,j,t} \to t \lb \alpha_k + \frac{y_{(t)} \alpha_k}{\alpha_k - 1} \rb 
    = t\alpha_k + \frac{y \alpha_k}{\alpha_k - 1} =: \phi_t (\alpha_k) \textnormal{ almost surely} .
    \label{def_phi}
\end{align}
For fixed $t\in [0,1]$, the fluctuations around the quantity  $$\lambda_{k,t}= \phi_t (\alpha_k) $$ were first studied in \cite{paul2007asymptotics, 
bai2008central}
and are seen to be of the order $1/\sqrt{n}.$
In this work, we are interested in the weak convergence of
\begin{align} \label{conv_process}
    \sqrt{n} \lb \lambda_{n,j,t} - \lambda_{k,t} \rb , \quad t\in [t_0,1],
\end{align}
considered as a process in a certain function space. For the formulation of our main result, we need the spectral decomposition of the spiked covariance matrix
\begin{align*}
    \bfSigma = \mathbf{U} \operatorname{diag}(\alpha_1 \bfI_{n_1}, \ldots, \alpha_K \bfI_{n_K} ) \mathbf{U}^\top, 
\end{align*}
where $\mathbf{U}$ is a $M \times M$ is an orthogonal matrix. Here, $\operatorname{diag}(\bfA)$ denotes the diagonal matrix, which has the same diagonal as the matrix $\bfA$.
The limiting distribution of \eqref{conv_process} depends on the eigenvalues of the random matrix
\begin{align*}
    \tilde \bfR_t (\lambda_{k,t} ) 
    = \mathbf{U}^\top \bfR_t (\lambda_{k,t}) \mathbf{U}, 
\end{align*}
where $\bfR_t(\lambda_{k,t})$ is a symmetric random matrix with centered Gaussian entries, and will be defined later on (see \eqref{def_R}).

\section{Main result on the process of eigenvalues}
\label{sec_main_results_eigenval}
 
The limiting distribution of $(\lambda_{n,j,t})$ will depend on quantities related to the \MP law. For $\lambda \notin I_t= [( 1-\sqrt{y_{(t)}} ) ^2, ( 1 + \sqrt{y_{(t)}} ) ^2], y_{(t)} = y/t,  t\in (0,1] $, we define
\begin{align*}
    m_{3,t}(\lambda) 
    = \int \frac{x}{(\lambda - x)^2} dF_{y_{(t)}} (x).
\end{align*}
For $\lambda = \lambda_{k,t}$, we give an explicit formula for 
$m_{3,t}(\lambda_{k,t})$ in Lemma \ref{lem_formula_m_lambda}. 
Here, $\tilde{F}_{y_{(t)}}$ denotes the standard \MP law with unit variance and scale parameter $y_{(t)}= y/t = \lim_{n\to\infty} p/ \nt$ and $F_{y_{(t)}}(\cdot) = \tilde{F}_{y_{(t)}}(\cdot/t)$ is a transformation of the \MP law which adjusts for the factor $1/n$ in the definition of $\bfS_{2,t}$ (instead of $1/\nt).$ That is, the limiting spectral distribution of $\bfS_{2,t}$ is $F_{y_{(t)}}$.
The function $m_t(\lambda) = m_t$ denotes the Stietljes transform of  \MP law $F^{y_{(t)}}$ and satisfies
\begin{align} \label{eq_stieltjes_MP}
\frac{1}{m_t} = t - y - \lambda - \lambda y m_t. 
\end{align}
Its dual Stieltjes transform $\su_t(\lambda) = \su_t$ satisfies
\begin{align}
    \label{eq_m_mu}
    \su_t = - \frac{1 - y_{(t)}}{\lambda} + y_{(t)} m_t. 
\end{align}
We are now in the position to state our main result concerning the joint convergence of the sample eigenvalues corresponding to several spiked eigenvalues; compare \cite[Section 3.2]{qinwen2014joint} for the case $t=1.$ In the following discussion, the symbol $\rightsquigarrow$ denotes the weak convergence of stochastic processes. 
\begin{theorem} \label{thm_joint_conv}
     Let $t_0\in (0,1)$ and $\mathcal{K} \subset \{ 1, \ldots, K\}$, such that
 $\alpha_k \notin [ 1 - \sqrt{y_{(t_0)}}, 1 + \sqrt{y_{(t_0)}} ] $ for every $k\in \mathcal{K}.$ 
    Then, we have 
    \begin{align*}
        \left\{ \sqrt{n} \lb \lambda_{n,j,t} - \lambda_{k,t} \rb
        \right\}_{k \in \mathcal{K}, j\in J_k, t\in [t_0,1]} \rightsquigarrow
        \left\{ \lambda_j \lb \frac{1}{y m_{3,t}(\lambda_{k,t}) \alpha_k} \tilde{\bfR}_{t}^{(kk)} (\lambda_{k,t}) \rb \right\}_{k \in \mathcal{K}, j\in J_k, t\in [t_0,1]}  
    \end{align*}
    in $(\ell^\infty([t_0,1]))^{\sum_{k \in \mathcal{K}} n_k}$. Here, for $k\in\mathcal{K},$ 
    $\tilde{\bfR}_{t}^{(kk)} (\lambda_{k,t}) =  (r_{ij}^{(k,t)})_{1 \leq i,j\leq n_k} $ denotes the $n_k \times n_k$ matrix that constitutes the $k$th diagonal block of $\tilde{\bfR}_t(\lambda_{k,t})$, 
    and $(\tilde\bfR_t (\lambda_{k,t}) )_{k \in \mathcal{K}, t\in [t_0,1]}$ is a process of $\lb \sum_{k\in\mathcal{K}} n_k \rb \times \lb \sum_{k\in\mathcal{K}} n_k \rb $ matrices. The entries $r_{ij}^{(k,t)}$ are zero-mean Gaussian random variables with covariance kernel 
    \begin{align*}
    \cov \lb r_{ij}^{(k',s)}, r_{ml}^{(k,t)} \rb 
    & = \omega_{s,t, k', k} \lb \E [ \xi(i) \xi(j) \xi(m) \xi(l) ]
    - \Sigma_{ij} \Sigma_{ml} \rb 
    \\ & + (\theta_{s,t, k', k} - \omega_{s,t, k', k} ) \lb 
    \E [ \eta(i) \eta(m)] \E [ \eta (j) \eta (l)] 
    + \E [ \eta(i) \eta(l)] \E [ \eta (j) \eta (m)] 
    \rb \\
    & = \omega_{s,t, k', k} \lb \E [ \xi(i) \xi(j) \xi(m) \xi(l) ]
    - \Sigma_{ij} \Sigma_{ml} \rb 
     \\ & \quad 
    + (\theta_{s,t, k', k} - \omega_{s,t, k', k} ) \lb 
    \delta_{im} \delta_{jl} + \delta_{il} \delta_{jm} 
    \rb 
\end{align*}
for $1 \leq i,j,m,l \leq M$, $s,t\in [t_0,1], k, k' \in \mathcal{K}.$  
The quantities $\omega_{s,t, k', k}$ and $\theta_{s,t, k', k}$ are given by 
\begin{align*}
     \omega_{s,t, k', k} & = 
     \tau_{s,t, k', k} + \psi_{s,t, k', k},
    \\ 
     \theta_{s,t, k', k}  &= \tau_{s,t, k', k} + 
     ( s \wedge t ) \zeta_{s,t,k,k'} \big\{ y^2 m_s(\lambda_{k',s}) m_t (\lambda_{k,t}) + \kappa_{s,t,k,k'} \zeta_{s,t,k,k'} \big\}
\end{align*} 
where
\begin{align*}
    \tau_{s,t,k,k'} &=  s \wedge t 
    -   (s \wedge t) y \left\{ \frac{1}{t} \lb 1 + \lambda_{k,t} m_{t} (\lambda_{k,t}) \rb 
    + \frac{1}{s} \lb 1 + \lambda_{k',s} m_{s} (\lambda_{k',s}) \rb 
    \right\}, \\
    \psi_{s,t,k,k'} & =\frac{ ( s \wedge t ) y^2 \lambda_{k',s} m_{s} (\lambda_{k',s}) \lambda_{k,t} m_{t} (\lambda_{k,t})  }{ \lambda_{k',s} \lambda_{k,t} \lb 1+  y m_{s} (\lambda_{k',s}) \rb \lb 1 + y_{} m_t(\lambda_{k,t})  \rb  } 
   , \\
    \kappa_{s,t,k,k'} & =  \frac{ (s \wedge t ) y m_s(\lambda_{k',s}) m_t(\lambda_{k,t})}{1 - ( s \wedge t ) y \lambda_{k,t} \lambda_{k',s} \su_s (\lambda_{k',s}) \su_t (\lambda_{k,t}) m_s (\lambda_{k',s}) m_t (\lambda_{k,t}) } , \\
    \zeta_{s,t,k,k'} &= 1 
            + y \lambda_{k',s} \su_s (\lambda_{k',s}) m_s (\lambda_{k',s})
            + y \lambda_{k,t} \su_t (\lambda_{k,t})  m_t (\lambda_{k,t}) \\ & \quad 
            + y^2 \lambda_{k',s} \lambda_{k,t} \su_s (\lambda_{k',s}) \su_t (\lambda_{k,t}) m_s (\lambda_{k',s}) m_t (\lambda_{k,t}).
\end{align*}
\end{theorem}

\begin{remark} \label{rem}
{ \rm 
    \begin{enumerate} 
        \item The limiting process in Theorem \ref{thm_joint_conv} is in general non-Gaussian, which aligns with the result in \cite{bai2008limit} for the case $t=1$. If and only if the spiked eigenvalue $\alpha_k$ of $\bfSigma$ is of simple multiplicity, the limiting process of the corresponding sample eigenvalue $\lambda_{n,j,t}$ is Gaussian. 
          \item          
       The terms $\omega_{s,t,k',k}$ and $\theta_{s,t,k'k}$ determine the covariance structure and thus, the nature of the limiting process. 
         These quantities arise from an application of a functional central limit theorem for sesquilinear forms (see Theorem \ref{thm_functional_clt}). To gain further insight, let us assume for convenience that $k=k'$. Then, the quantities $\theta_{s,t}=\theta_{s,t,k,k}$ and $\omega_{s,t}=\omega_{s,t,k,k}$  occur as the limits of $(1/n) \tr ( \bfB_{n,t} \circ \bfB_{n,s} )$ and $(1/n) \tr ( \bfB_{n,t} \bfB_{n,s} )$, where $\bfB_{n,s}$ and $\bfB_{n,t}$ include resolvent-type matrices and $\circ$ denotes the Hadamard product (for more details, see \eqref{omega_limit} and \eqref{theta_limit}). Due to the intricate structure induced by the superposition of the samples $\bfx_1, \ldots, \bfx_{\ns}$ and $\bfx_1, \ldots, \bfx_{\nt}$, computing these terms is one of the key technical challenges addressed in this paper.
         For more details on this step, please refer to the outline of the proof in Section \ref{sec_proof_outline}.
    \item 
   Regarding the entries of the matrix $\tilde{\mathbf{R}}^{(kk)}(\lambda_{k,t})$, their variances are given by,
\begin{align*}
    \Var ( r_{ij}^{(k,t)})
    = \omega_{t,t,k,k} \lb \E [ \xi(i)^2 \xi(j)^2] - 2\Sigma_{ij}^2 - \Sigma_{ii} \Sigma_{jj} \rb + \theta_{t,t,k,k} \lb \Sigma_{ij}^2 + \Sigma_{ii} \Sigma_{jj} \rb  ,
\end{align*}
which simplifies to
\begin{align*}
     \Var ( r_{ii}^{(k,t)})
      = \omega_{t,t,k,k} \lb \E [ \xi(i)^4] - 3\Sigma_{ii}^2 \rb + 2 \theta_{t,t,k,k} \Sigma_{ii}^2
\end{align*}
for a diagonal element. 
If the kurtosis of $\xi(i)$ equals three, that is, $\E[ \xi(i)^4]/\Sigma_{ii}^2 =3$, then we have
then $\Var ( r_{ii}^{(k,t)})
      =  2 \theta_{t,t,k,k} \Sigma_{ii}^2$. 
      
\item 
In general, the processes of sample eigenvalues
        \begin{align*}
            \left\{ \sqrt{n} \lb \lambda_{n,j,t} - \lambda_{k,t} \rb
        \right\}_{j\in J_k, t\in [t_0,1]}, \quad 1 \leq k \leq K,
        \end{align*}
        are not asymptotically independent, 
         which has been found in \cite{qinwen2014joint} for the finite-dimensional case $t=1$ (see Remark 3.4, Remark 3.5 and Section 3.2.2 in their work).
        However, if 
         \begin{align} \label{eq_fac_mean}
             \E [ \xi(i) \xi(j) \xi(m) \xi(l) ]
             = \E [ \xi(i) \xi(j) ] \E[ \xi(m) \xi(l) ], \quad i,j \in J_k, m,l\in J_{k'}
         \end{align}
         for some $1 \leq k \neq k' \leq K$, then $r_{ij}^{(k',s)}$ and $r_{ml}^{(k,t)}$ are independent, and so are 
         \begin{align*}
              \left\{ \sqrt{n} \lb \lambda_{n,j,t} - \lambda_{k,t} \rb
        \right\}_{j\in J_k, t\in [t_0,1]} 
        \textnormal{ and }
         \left\{ \sqrt{n} \lb \lambda_{n,j,t} - \lambda_{k',t} \rb
        \right\}_{j\in J_{k'}, t\in [t_0,1]}.
         \end{align*} 
    Note that \eqref{eq_fac_mean} is satisfied if the components of $\xi$ are independent. 

\item  
Let $t\in [t_0,1]$ and $k\in \{1, \ldots, K\}$ such that the corresponding block $(\Sigma_{ij})_{i,j\in J_k}=\sigma_k^2 \bfI_{n_k}$ of the population covariance matrix is spherical with some variance $         \sigma_k^2>0$. If the kurtosis of $\xi(i)$ equals three, that is, $\E[ \xi(i)^4]/\Sigma_{ii}^2 =3$, for every $i\in J_k$, then the corresponding block $\tilde{\bfR}_{t}^{(kk)}$ appearing in the limiting process in Theorem \ref{thm_main} can be interpreted as a rescaled GOE
(Gaussian Orthogonal Ensemble). 
      \end{enumerate} 
    }
\end{remark}

\section{Change-point test for large covariance matrices} \label{sec_test}

In this section, we develop statistical tests to monitor changes in the covariance structure in terms of changes in magnitude of their spiked eigenvalues. Structural changes of this type, therefore, are hard to be detected by tests based 
aggregative statistics, such as functions of tracial moments of sequential sample covariances. A more effective strategy is to focus attention on the  leading eigenvalues of the sequential sample covariances. In Section \ref{sec_known_spikes}, we explore a test for structural breaks in a scenario where the initial spiked eigenvalues are specified. Furthermore, we present an adaption of our methodology in Section \ref{sec_unknown_spikes} through a plug-in approach that incorporates estimation of the unknown population spiked eigenvalues, and functions thereof. 
Our theoretical statistical guarantees are illustrated by a simulation study in Section \ref{sec_sim}. 

For this purpose, let $\bfy_i = (\xi_i^\top, \eta_i^\top)$, $1 \leq i \leq n,$ be independent random vectors with covariance matrix 
\begin{align*}
    \cov(\bfy_i) = \begin{pmatrix}
 \bfSigma_i & \mathbf{0} \\ 
 \mathbf{0} & \bfI_p
    \end{pmatrix}.
\end{align*}
Note that the random vectors $\xi_1, \ldots, \xi_n, \eta_1, \ldots, \eta_n$ and the spiked covariances $\bfSigma_1, \ldots, \bfSigma_n$ are modelled as in Section \ref{sec_spike_model} with the only difference that $\bfSigma_i$ may depend on $1 \leq i \leq n$ (but not on $n$). 

\subsection{Testing for structural breaks with known initial spiked eigenvalues} \label{sec_known_spikes}

Let $\bfSigma_0$ be a given covariance matrix, and we assume that the eigenvalues of $\bfSigma_i$ are of simple multiplicity for $1 \leq i \leq M$. 
Then, the test problem can be formulated as follows:
\begin{align}
    H_0: & \bfSigma_1 = \ldots = \bfSigma_n = \bfSigma_0 
    \textnormal{ vs. } \label{hypothesis} \\ 
    H_1: & ~\exists~ t^\star \in (t_0,1], 1 \leq k \leq M:   \lambda_k (\bfSigma_1) = \ldots = \lambda_k(\bfSigma_{\lfloor nt^\star \rfloor } ) = \lambda_k(\bfSigma_0) \neq \lambda_k(\bfSigma_{\lfloor nt^\star \rfloor + 1 } ) = \ldots = \lambda_k(\bfSigma_n). \nonumber 
\end{align}

We propose two different test statistics:
    \begin{align}
         M_n & = \sup_{t \in [t_0,1], 1 \leq k \leq M} 
        n \left\{  \lambda_{n,k,t} - \phi_{n,t} ( \lambda_k( \bfSigma_0 ) )  \right\}^2, \label{def_max_stat} \\
        S_n & = \sup_{t \in [t_0,1]} \sum_{k=1}^M
        n \left\{  \lambda_{n,k,t} - \phi_{n,t} ( \lambda_k( \bfSigma_0 ) )  \right\}^2, \label{def_sum_stat} 
    \end{align}
    where
    \begin{align*}
        \phi_{n,t}(\alpha) = t\alpha + \frac{y_n \alpha}{\alpha - 1}.
    \end{align*}
    In Figure \ref{fig_hist1}, the distributions of $\log M_n$ and $\log S_n$ are illustrated by histograms for different values of $n$ and $p$. 
    The following theorem provides the asymptotic behavior of $M_n$ and $S_n$ under $H_0$ and $H_1$.  
    
\begin{theorem} \label{thm_test_statistics}
Let $t_0\in (0,1)$, $y_n - y = o(n^{-1/2})$ and assume that the eigenvalues of $\bfSigma_1, \ldots, \bfSigma_n$ are outside the interval $ [ 1 - \sqrt{y_{(t_0)}}, 1 + \sqrt{y_{(t_0)}} ] $.
    Under $H_0$, we have for $n\to\infty$
    \begin{align*}
        M_n & \cond \sup_{t\in[t_0,1], 1 \leq k \leq M} G_{k,t}^2, \\
        S_n & \cond \sup_{t\in[t_0,1]} \sum_{k=1}^M G_{k,t}^2,
    \end{align*}
    where $(G_{k,t})_{t\in [t_0,1]}$ are centered Gaussian processes for $1 \leq k \leq M$. More precisely, we have
    \begin{align*}
        G_{k,t} =
        \frac{ r_{kk}^{(k,t)} }{1 + y m_{3,t}(\lambda_{k,t}) \alpha_k} 
        , \quad 1 \leq k \leq M, ~ t\in [t_0,1],
    \end{align*}
    where the covariance kernel of $r_{kk}^{(k,t)}$ is given in Theorem \ref{thm_joint_conv}. 
    Under $H_1$, we have $\min ( M_n, S_n ) \conp \infty$.
\end{theorem}
Under $H_0$, the previous result is a direct consequence of Theorem \ref{thm_joint_conv} and the continuous mapping theorem (see \cite[Theorem 1.3.6]{vandervaart1996}), and for the assertion under $H_1$, we note that if $\phi_t ( \lambda_k( \bfSigma_0 ) ) = \phi_t ( \lambda_k( \bfSigma_n ) )$ for all $t \in [t_0,1], ~ 1 \leq k \leq M,$ then $\lambda_k( \bfSigma_0 ) = \lambda_k( \bfSigma_n ) $. Thus, we omit a detailed proof for the sake of brevity. 
\begin{figure}[!ht]
    \centering
             \includegraphics[width=0.49\columnwidth, height=0.5\textheight, keepaspectratio]
             {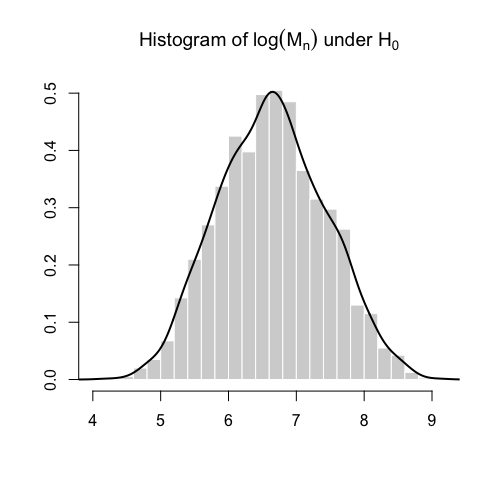}
             \includegraphics[width=0.49\columnwidth, height=0.5\textheight, keepaspectratio]
             {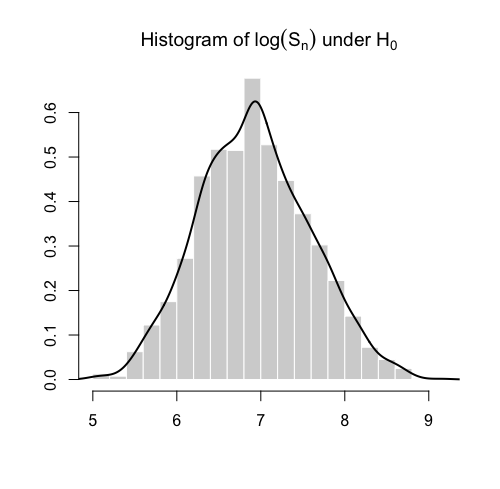} 
             \includegraphics[width=0.49\columnwidth, height=0.5\textheight, keepaspectratio]
             {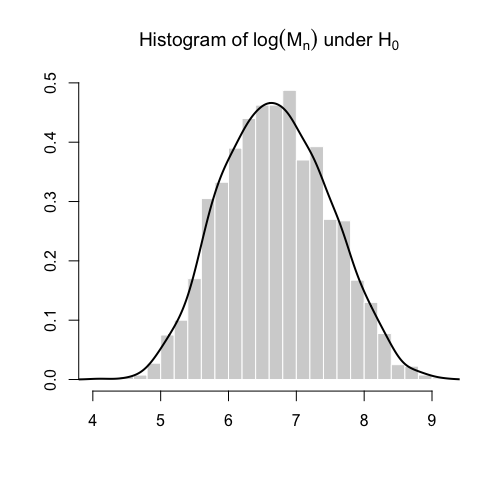}
             \includegraphics[width=0.49\columnwidth, height=0.5\textheight, keepaspectratio]
             {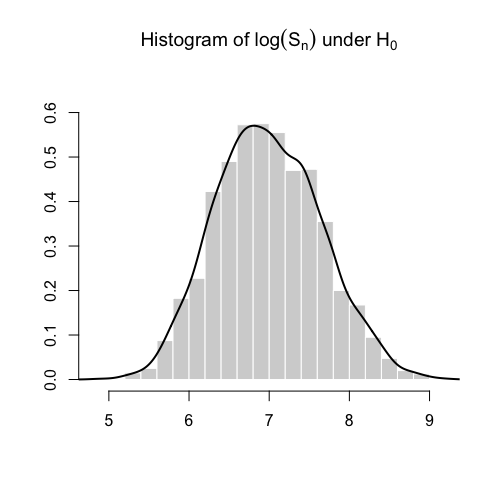} 
         \caption{   \it  Histograms of the statistics $\log M_n$ given in \eqref{def_max_stat} (left column) and $\log S_n$ given in \eqref{def_sum_stat} (right column) under $H_0$ given in \eqref{alt1} ($\delta = 0$) based on $2000$ simulation runs,   where $n=400, ~ p=200$ (first row) and $n=500, ~ p=300$ (last row).  }
    \label{fig_hist1}
    \end{figure}

\begin{remark}  \label{remark_test}
{\rm
    Let $q_{m,\alpha}$ be the $\alpha$-quantile of $\sup_{t\in[t_0,1], 1 \leq k \leq M} G_{k,t}^2$. Then, Theorem \ref{thm_test_statistics} shows that the test, which rejects $H_0$, whenever
    \begin{align} \label{test_max}
      M_n > q_{m,1 - \alpha} 
    \end{align}
    attains asymptotically the prescribed level $\alpha\in(0,1)$ under $H_0$, and is consistent for $H_1$. An analogous result holds true for the test constructed upon the sum-type statistic $S_n$. More precisely, let $q_{s,\alpha}$ be the $\alpha$-quantile of $\sup_{t\in[t_0,1]} \sum_{k=1}^M G_{k,t}^2$, and reject $H_0$ in \eqref{hypothesis} whenever 
    \begin{align} \label{test_sum}
      S_n > q_{s,1 - \alpha} .
    \end{align}
}
\end{remark}
  
\begin{remark}  \label{remark_multiplicity}    
{\rm 
    We emphasize that the assumption concerning the single multiplicity of the spiked eigenvalues is solely for technical convenience. A similar test could be constructed when allowing for larger multiplicities, resulting in a different distribution of $(G_{k,t})_{k,t}$ (see also Remark \ref{rem}). 
}
\end{remark}

In applications, the number of spiked eigenvalues $M=K$ is usually unknown and can be estimated from the data, e.g., using the methodology developed in \cite{li_et_al_2020, ke2023estimation}.
    Moreover, if the concrete values $\lambda_1(\bfSigma_0), \ldots, \lambda_K(\bfSigma_0)$ are not accessible for the user, we may extend our method to a test for the null hypothesis $\bfSigma_1 = \ldots = \bfSigma_n$ without specifying a prescribed matrix $\bfSigma_0$, which is the content of the following section.

\subsection{Estimating the spiked eigenvalues} \label{sec_unknown_spikes}

In this section, we demonstrate how to adapt our methodology to the case that the covariance matrix $\bfSigma_0$ is not prespecified by the user.  That is, we interested in the testing problem 
\begin{align}
    H_0: & \bfSigma_1 = \ldots = \bfSigma_n
    \textnormal{ vs. } \label{hypothesis2} \\ 
    H_1: & \exists t^\star \in (t_0,1], 1 \leq k \leq M:   \lambda_k (\bfSigma_1) = \ldots = \lambda_k(\bfSigma_{\lfloor nt^\star \rfloor } )  \neq \lambda_k(\bfSigma_{\lfloor nt^\star \rfloor + 1 } ) = \ldots = \lambda_k(\bfSigma_n). \nonumber 
\end{align}
For this purpose, we assume that we have access to an initial sample of i.i.d. random vectors $\bfy_1^{(0)}, \ldots, \bfy_n^{(0)}$, which are known to be generated under $H_0$, and independent of the actual sample $\bfy_1, \ldots, \bfy_n$ given in the previous section. Let $\lambda_{n,1,t}^{(0)} \geq \ldots \geq \lambda_{n,p,t}^{(0)} $ denote the sample eigenvalues corresponding to the initial sample $\bfy_1^{(0)}= (\xi_1^{(0)},  \eta_1^{(0)})^\top , \ldots, \bfy_n^{(0)}=(\xi_n^{(0)},  \eta_n^{(0)})^\top$. 

As an estimate for the population spiked eigenvalue $\alpha_k, ~ 1 \leq k \leq M,$ we use 
    \begin{align} \label{def_hat_alpha}
  \hat \alpha_{k} = \frac{ \lambda_{n,k,1}^{(0)} + 1 - y_n + \sqrt{(\lambda_{n,k,1}^{(0)} + 1 - y_n)^2 - 4 \lambda_{n,k,1}^{(0)}}}{2},
    \end{align}
    which solves the equation
    \begin{align*}
        \phi_{n,1}(\hat\alpha_k) = \lambda_{n,k,1}^{(0)}. 
    \end{align*}
    Using \eqref{def_phi}, it can be shown that $\hat\alpha_k$ is a consistent estimator for $\alpha_k$, that is, $\hat\alpha_k \conp \alpha.$
Note that our model operates under the assumption that $\cov(\eta) = \bfI$. In cases where the eigenvalues of the background noise are non-unit and unknown, we recommend utilizing the estimators suggested in \cite{passemier2017estimation, bai2012estimation}.
   
Then, we propose to use the test statistics 
 \begin{align}
        \hat M_n & = \sup_{t \in [t_0,1], 1 \leq k \leq M}  n   \lb \lambda_{n,k,t} - \lambda_{n,k,t}^{(0)} \rb ^2 ,
        \label{def_max_stat2} \\
        \hat S_n & = \sup_{t \in [t_0,1]} \sum_{k=1}^M
         n   \lb \lambda_{n,k,t} - \lambda_{n,k,t}^{(0)} \rb ^2.
    \end{align}  

Analogously to Theorem \ref{thm_test_statistics}, we have the following result on the asymptotic behaviour of $\hat M_n$ and $\hat S_n.$
    \begin{theorem} \label{thm_test_statistics2}
Let $t_0\in (0,1)$, $y_n - y = o(n^{-1/2})$ and assume that the eigenvalues of $\bfSigma_1, \ldots, \bfSigma_n$ are outside the interval $ [ 1 - \sqrt{y_{(t_0)}}, 1 + \sqrt{y_{(t_0)}} ] $.
    Under $H_0$ given in \eqref{hypothesis2}, we have for $n\to\infty$
    \begin{align*}
        \hat M_n & \cond \sup_{t\in[t_0,1], 1 \leq k \leq M} H_{k,t}^2, \\
        \hat S_n & \cond \sup_{t\in[t_0,1]} \sum_{k=1}^M H_{k,t}^2,
    \end{align*}
    where $(H_{k,t})_{t\in [t_0,1]}$ are centered Gaussian processes for $1 \leq k \leq M$. More precisely, we have
    \begin{align*}
        H_{k,t} =  \frac{ r_{kk}^{(k,t)} - r_{kk}^{(0,k,t)} }{1 + y m_{3,t}(\lambda_{k,t}) \alpha_k} 
        , \quad 1 \leq k \leq M, ~ t\in [t_0,1],
    \end{align*}
    where the covariance kernel of $r_{kk}^{(k,t)}$ is given in Theorem \ref{thm_joint_conv}, and $r_{kk}^{(0,k,t)}$ denotes an independent copy of $r_{kk}^{(k,t)}$.
    Under $H_1$, it holds that $\min ( \hat M_n, \hat S_n ) \conp \infty$.
\end{theorem}

\begin{remark} \label{remark_quantile}
{\rm 
We can construct a test for the problem \eqref{hypothesis} which holds asymptotically its level and is consistent,  by rejecting for large values of $\hat M_n$ and $\hat S_n$, respectively, similarly to Remark \ref{remark_test}. 
    To find the quantiles of the limiting processes in Theorem \ref{thm_test_statistics2} numerically, we propose to use the following quantities
    \begin{align*}
        \hat\Sigma_{ij} & = \frac{1}{n} \sum_{l=1}^n \xi_l^{(0)} (i) \xi_l^{(0)} (j), \\
        \hat \E [ \xi(i) \xi(j) \xi(m) \xi(l) ] & = \frac{1}{n} \sum_{l=1}^n \xi_l^{(0)} (i) \xi_l^{(0)} (j) \xi_l^{(0)} (m) \xi_l^{(0)} (l), \\
        m_{3,n,t}(\lambda ) 
 & = \int \frac{x}{(\lambda - x)^2} dF_{y_{\nt}} (x).
\end{align*}
 Here, $\xi_l^{(0)}(i)$ denotes the $i$th coordinate of $\xi_l^{(0)}$ for $1 \leq i \leq M, ~ 1 \leq l \leq n.$
Note that $\hat\Sigma_{ij}$, $\hat \E [ \xi(i) \xi(j) \xi(m) \xi(l) ]$, and $m_{3,n,t}(\lambda_{n,k,1}^{(0)})$ are consistent estimators of  $\Sigma_{ij}$, $\E [ \xi(i) \xi(j) \xi(m) \xi(l) ]$, and $m_{3,t}(\lambda_{k,t})$, respectively. Then, the quantiles of $\sup_{t\in[t_0,1], 1 \leq k \leq M} H_{k,t}^2$ and $\sup_{t\in[t_0,1]} \sum_{k=1}^M H_{k,t}^2$, respectively, can be approximated by using the quantiles of the corresponding Gaussian processes when replacing all the unknown quantities in the covariance kernel by the proposed estimators.
}
\end{remark}

\begin{remark} \label{remark_prelim_sample_size}
We would like to emphasize that the assumption on the initial sample size $N$ coinciding with the actual sample size $n$ is solely for convenience of exposition. Indeed, we may allow for a general size $N$ of the initial sample $\bfy_1^{(0)}, \ldots, \bfy_N^{(0)}$, which may not necessarily equal $n$, as long as $p/N \to y_0 \in (0,1). $ In this case, we get an estimator $\hat\alpha_{k,N}$ by solving
    \begin{align*}
        \phi_{N,1}(\hat \alpha_{k,N}) = \lambda_{N,k,1}^{(0)},
    \end{align*}
which leads similarly as in \eqref{def_hat_alpha} to 
\begin{align*}
  \hat \alpha_{k,N} = \frac{ \lambda_{N,k,1}^{(0)} + 1 - y_N + \sqrt{(\lambda_{N,k,1}^{(0)} + 1 - y_n)^2 - 4 \lambda_{N,k,1}^{(0)}}}{2},
~
\end{align*} 
and satisfies $\hat\alpha_{k,N} \conp \alpha_k$ as $N\to\infty$. 

\end{remark}
   
\subsection{Simulation study} \label{sec_sim}

	We conclude this section with a small simulation study illustrating the finite-sample properties of two tests in \eqref{test_max} and \eqref{test_sum} based on $M_n$ and $S_n$ in the presence of $M=3$ spiked eigenvalues $\alpha_1, \alpha_2, \alpha_3$. For this purpose, we generated $p$-dimensional centered normally distributed data with different covariance structures. 
		To be precise, we consider the  the alternatives 
    \begin{align}
        \bfSigma_1 & = \ldots = \bfSigma_{\lfloor nt^\star \rfloor} = \operatorname{diag}(\alpha_1, \alpha_2, \alpha_3),  ~ 
	    \bfSigma_{\lfloor nt^\star \rfloor+ 1}  = \ldots
	    = \bfSigma_n = \operatorname{diag}(\alpha_1 + \delta , \alpha_2 , \alpha_3),
        \label{alt1} \\ 
        \bfSigma_1 & = \ldots = \bfSigma_{\lfloor nt^\star \rfloor} = \operatorname{diag}(\alpha_1, \alpha_2, \alpha_3),  ~ 
	    \bfSigma_{\lfloor nt^\star \rfloor+ 1}  = \ldots
	    = \bfSigma_n = \operatorname{diag}(\alpha_1, \alpha_2 + \delta, \alpha_3),
        \label{alt2} \\
         \bfSigma_1 & = \ldots = \bfSigma_{\lfloor nt^\star \rfloor} = \operatorname{diag}(\alpha_1, \alpha_2, \alpha_3),  ~ 
	    \bfSigma_{\lfloor nt^\star \rfloor+ 1}  = \ldots
	    = \bfSigma_n = \operatorname{diag}(\alpha_1 + \delta, \alpha_2 + \delta, \alpha_3),
        \label{alt3}
    \end{align}
    where $\delta \geq 0$ determines the deviation from  the null hypothesis. Note that the choice $\delta = 0$ corresponds to the null hypothesis 
	\eqref{hypothesis}. Moreover, we set $\alpha_3 = 1+\sqrt{y_{(t_0)}} + 1, ~\alpha_2 = \alpha_3 +3, \alpha_1 = \alpha_2 +10.$ To begin with, we consider the case, where the spiked eigenvalues are known.

In Figures \ref{fig_alt1} and \ref{fig_alt2}, we present the observed rejection rates of the tests \eqref{test_max} and \eqref{test_sum} for various alternative scenarios and different combinations of values for $n$ and $p$. Throughout the analysis, we always set the value of $t_0$ as $0.1$.
Regarding the change point, we use the value $t^\star = 0.5(1-t_0) = 0.6$. All the reported findings are based on a total of $2000$ simulated runs. The horizontal gray line in each figure serves as a reference to the significance level $\alpha = 5\%$ under consideration. The x-axis captures the variable $\delta$, while the y-axis represents the corresponding empirical rejection rate.

In all cases under consideration, we observe a good approximation of the nominal level $\alpha = 5 \%$ under the null hypothesis ($\delta=0$). In a scenario where a change occurs in the first eigenvalue $\alpha_1$ (alternative \eqref{alt1} in Figure \ref{fig_alt1}), we observe very similar power curves for both tests \eqref{test_max} and \eqref{test_sum}. 
However, when a change takes place in the second eigenvalue (alternative \eqref{alt2} in Figure \ref{fig_alt2}), the test based on the sum-type statistic \eqref{test_sum} outperforms in detecting the signal $\delta > 0$. 
In simpler terms, the max-type statistic is less sensitive to alterations in the second-largest eigenvalue when compared to changes in the leading eigenvalue. This disparity between the performance of \eqref{test_max} and \eqref{test_sum} can be attributed to the larger fluctuations associated with the largest spiked eigenvalue $\alpha_1$, leading to the dominance of the term $n \left\{ \lambda_{n,1,t} - \phi_t(\lambda_1(\mathbf{\Sigma}_0)) \right\}^2$ in \eqref{def_max_stat}, while this effect is less pronounced for the sum-type statistic \eqref{def_sum_stat}.
By comparing Figures \ref{fig_alt1} and \ref{fig_alt2}, we can observe that detecting a change in the second spiked eigenvalue poses a greater challenge, as evidenced by the slower increase in power. 

Similar considerations can be conducted in the case where the spiked eigenvalues $\alpha_1, \alpha_2, \alpha_3$ are unknown and are estimated through $\hat\alpha_1, \hat\alpha_2, \hat\alpha_3$ based on an initial sample as described in Section \ref{sec_unknown_spikes}.  In Figure \ref{fig_unknown_spikes}, we present the empirical rejection rate under alternatives \eqref{alt1} (left panel) and \eqref{alt3} (right panel) using the same setup as above. Under the null, both tests tend to be slightly conservative. Comparing the left and right panels in Figure \ref{fig_unknown_spikes}, it is apparent that detecting the alternative \eqref{alt1} (change in the first eigenvalue) is more challenging for a given signal $\delta$ compared to detecting \eqref{alt3} (changes in the first two eigenvalues). These numerical results reflect the underlying difference with respect to the signal strength. For instance, we have for the  Frobenius norm $\|\bfSigma_1 - \bfSigma_n\|_F = \delta$ under \eqref{alt1}, whereas $\|\bfSigma_1 - \bfSigma_n\|_F = \sqrt{2}\delta$ under \eqref{alt3}.
Moreover, the two proposed tests perform very similarly for detecting the alternative \eqref{alt1}. However, it is noteworthy, that the test based on $\hat M_n$ outperforms the one involving $\hat S_n$. In this scenario, the the test based on $\hat M_n$  gains it superior power from the fact that its fluctuations are dominated by the largest spiked eigenvalues, while this effect is less pronounced for $\hat S_n$. 

 \begin{figure}[!ht]
    \centering
             \includegraphics[width=0.49\columnwidth, height=0.35\textheight, keepaspectratio]{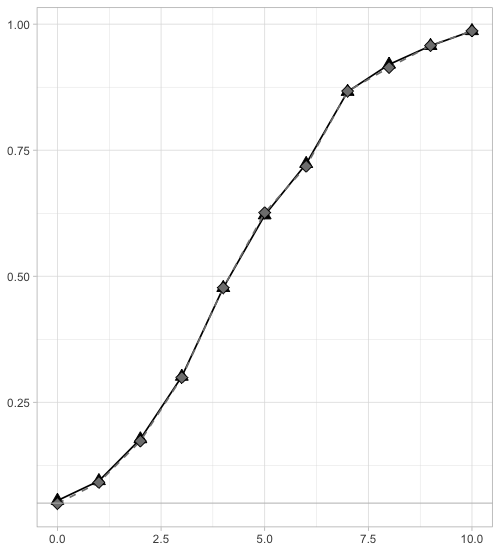}
             \includegraphics[width=0.49\columnwidth, height=0.35\textheight, keepaspectratio]{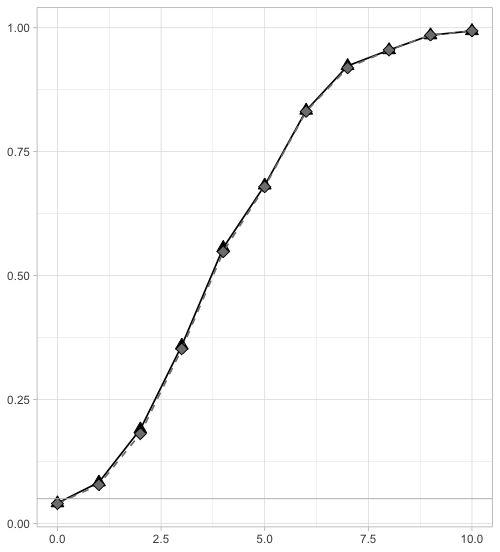}    
         \caption{   \it  Empirical rejection probabilities of the tests \eqref{test_max} (triangle and solid line) and \eqref{test_sum} (square and dotted line) under the null hypothesis ($\delta=0$) and the different alternatives given in \eqref{alt1}  for $\delta >0$ based on $2000$ simulation runs. For the left figure, we have $n=400, ~ p=200$ and for the right one, $n=500, p=300$.  }
    \label{fig_alt1}
         
    \end{figure}
     \begin{figure}[!ht]
    \centering
             \includegraphics[width=0.49\columnwidth, height=0.35\textheight, keepaspectratio]{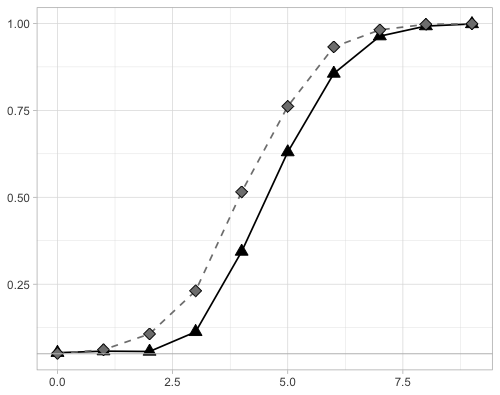}
             \includegraphics[width=0.49\columnwidth, height=0.35\textheight, keepaspectratio]{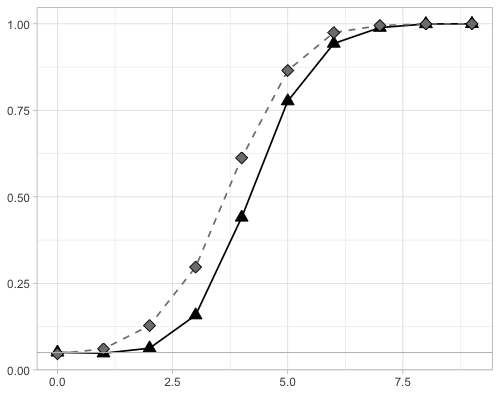}    
         \caption{   \it  Empirical rejection probabilities of the tests \eqref{test_max} (triangle and solid line) and \eqref{test_sum} (square and dotted line) under the null hypothesis ($\delta=0$) and the different alternatives given in \eqref{alt2}  for $\delta >0$ based on $2000$ simulation runs. For the left figure, we have $n=400, ~ p=200$ and for the right one, $n=500, p=300$.  }
    \label{fig_alt2}
         
    \end{figure}

   \begin{figure}[!ht]
    \centering
             \includegraphics[width=0.49\columnwidth, height=0.35\textheight, keepaspectratio]{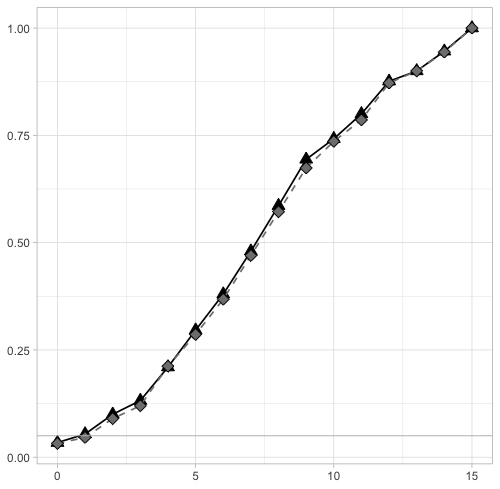}
             \includegraphics[width=0.49\columnwidth, height=0.35\textheight, keepaspectratio]{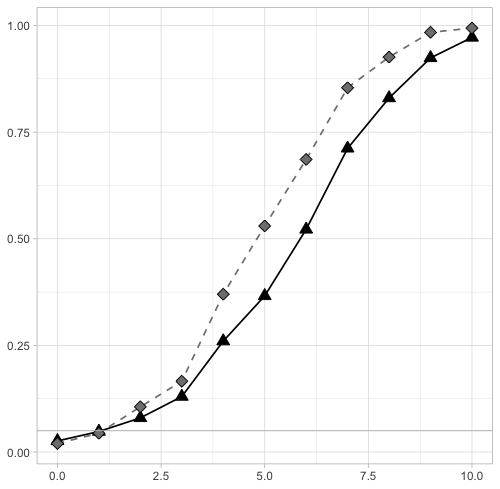}    
         \caption{  \it  Empirical rejection probabilities of the tests \eqref{test_max} (triangle and solid line) and \eqref{test_sum} (square and dotted line) in a scenario, where the spiked eigenvalues are estimated from an initial sample.  We consider the null hypothesis ($\delta=0$) and the different alternatives given in \eqref{alt1} (left panel) and \eqref{alt3} (right panel)  for $\delta >0$ based on $500$ simulation runs and set $n=400, ~ p=200$.  }
    \label{fig_unknown_spikes}
         
    \end{figure}
  
\newpage
    
\section{Proofs of Main Theorems}\label{sec_proof_main}

\subsection{Outline for the proof of Theorem \ref{thm_joint_conv}} \label{sec_proof_outline}

To prove weak convergence of the process  $\left\{ \sqrt{n} \lb \lambda_{n,j,t} - \lambda_{k,t} \rb
        \right\}_{j\in J_k, t\in [t_0,1]} $
of sequential sample spiked eigenvalues, we establish an immediate connection with a process involving certain matrices $\mathbf{R}_{n,t}(\lambda_{k,t})$ for $1 \leq k \leq K$, $n\in\mathbb{N}$, and $t\in [t_0,1]$.
That is, in the first step, we will show that the convergence of this process guarantees the convergence of the original process of sequential sample spiked eigenvalues. While the general idea of this step is similar to the previous works \cite{bai2008limit, qinwen2014joint, zhang2022asymptotic}, we emphasize that our sequential model requires the more challenging analysis of a process in an infinite-dimensional function space.  

From the outset, the choice of a suitable function space for $\left\{ \sqrt{n} \lb \lambda_{n,j,t} - \lambda_{k,t} \rb
        \right\}_{j\in J_k, t\in [t_0,1]} $ requires a nuanced decision. In the space of bounded functions $\ell^\infty$ equipped with the uniform norm, even simple stochastic processes can be non-measurable with respect to the topology induced by this norm. A classical approach is to restrict the underlying space further and opt for a coarser topology, for example by considering the Skorokhod space equipped with the Skorokhod metric (see \cite[Chapter 3]{billingsley1999}). However, it is well known that the uniform norm, as a mapping from the Skorokhod space equipped with the Skorokhod metric to $\R$, is not continuous, and thus, the convergence of the supremum of a process cannot be implied by the continuous mapping theorem. This poses a problematic issue, especially considering that many statistical applications, including our Section \ref{sec_test}, rely on the asymptotic behavior of the supremum of certain weakly convergent processes. Therefore, we turn to the modern theory of stochastic processes as outlined in \cite{vandervaart1996}, which circumvents the measurability problem through the notion of outer integrals. 

To investigate $(\bfR_{n,t}(\lambda_{k,t}))_{t\in [t_0,1], n\in\N}$ in $\ell^\infty$, we derive a general functional central limit theorem for sesquilinear forms (Theorem \ref{thm_functional_clt}), which may be of independent interest. When applying this result to $(\bfR_{n,t}(\lambda_{k,t}))_{t\in [t_0,1], n\in\N}$, the main challenges we encounter for the asymptotic analysis are twofold:
\begin{enumerate}
\item 
Calculating the covariance structure poses inherent difficulties compared to the finite-dimensional case ($t=1$) studied in previous works (e.g., \cite{bai2008limit, qinwen2014joint}). Specifically, the sequential framework considered in this work requires finding the limit of terms such as $(1/p) \tr ( \bfS_{2,s} \D_t (\lambda_{k,t} ) )$ or $(1/p) \tr (\bfS_{2,s} \D_s (\lambda_{k,s}) \bfS_{2,s} \D_t (\lambda_{k,t}))$ (see Lemma \ref{lem_tr_S_D} and Lemma \ref{lem_tr_S_D_S_D}). Here, $\D_s(\lambda_{k,s})= (\bfS_{2,s} - \lambda_{k,s} \bfI)\inv$ denotes the resolvent of $\bfS_{2,s}$. The case $s\neq t$ is particularly complex since, unlike the $s=t$ scenario, the aforementioned traces cannot be expressed as integrals with respect to the limiting spectral distribution of $\bfS_{2,s}$ or $\bfS_{2,t}$. Instead, more advanced tools are necessary to identify the limits, such as a resolvent decomposition as outlined in \cite{dornemann2021linear}.

\item 
In the presence of a sequential framework modelled in an infinite-dimensional function space, an additional challenge presents itself, which does not arise in the case $t=1$.
To verify the asymptotic tightness, appropriate concentration inequalities for the increments are needed (see proof of Theorem \ref{thm_R_tight} in Section \ref{sec_proof_R_tight}). Contrary to convergence of the finite-dimensional distributions of $(\bfR_{n,t}(\lambda_{k,t}))_{t\in [t_0,1], n\in\N}$, 
all estimates must hold uniformly over $t$, which necessitates
a finer level of control of the fluctuations. We address this issue by conditioning on certain events with probability approaching $1$, and establishing the required estimates on these events. 
\end{enumerate}

\subsection{Proof of Theorem \ref{thm_joint_conv} }

In order to simplify the presentation of this paper, we restrict our proofs to the case that $|\mathcal{K}|=1$. However, we emphasize that this simplification is solely for technical convenience, and the tools used for the proof in case $|\mathcal{K}|=1$ can be easily generalized for more sophisticated choices of $\mathcal{K}.$
Thus, we concentrate on the proof of the following result. 
\begin{theorem}\label{thm_main}
    Let $t_0\in (0,1)$ and $1 \leq k \leq K.$
    For every spiked eigenvalue $\alpha_k \notin [ 1 - \sqrt{y_{(t_0)}}, 1 + \sqrt{y_{(t_0)}} ] $,
    we have 
    \begin{align*}
        \left\{ \sqrt{n} \lb \lambda_{n,j,t} - \lambda_{k,t} \rb
        \right\}_{j\in J_k, t\in [t_0,1]} \rightsquigarrow
        \left\{ \lambda_j \lb \frac{1}{1 + y m_{3,t}(\lambda_{k,t}) \alpha_k} \tilde{\bfR}_{t}^{(kk)} (\lambda_{k,t}) \rb \right\}_{j\in J_k, t\in [t_0,1]}  
    \end{align*}
    in $(\ell^\infty([t_0,1]))^{n_k}$, where the $n_k \times n_k$ matrix $\tilde{\bfR}_{t}^{(kk)} (\lambda)$ is defined in Theorem \ref{thm_joint_conv}.
\end{theorem}
We define for $t\in [0,1], \lambda \in \R^+ \setminus I_{t} $ with $I_t= [a_{y_{(t)}}, b_{y_{(t)}}]$, the matrices  
\begin{align}
    \mathbf{A}_{n,t} & =  \mathbf{A}_{n,t} (\lambda) = \bfX_{2,t}^\top \lb \lambda \bfI - \bfX_{2,t} \bfX_{2,t}^\top \rb\inv \bfX_{2,t},\nonumber \\ 
    \bfR_{n,t} & = \bfR_{n,t} (\lambda) = \frac{1}{\sqrt{n}} \left\{ 
    \xi_{1:\nt} \lb \bfI + \bfA_{n,t} \rb \xi_{1:\nt}^\top 
    - \bfSigma \tr \lb \bfI + \bfA_{n,t} \rb
    \right\} .   \label{def_R}
\end{align}
The key result for proving Theorem \ref{thm_main} lies in investigating the asymptotic behavior of the process of matrices $\bfR_{n,t}(\lambda_{k,t})$ and is provided in the following theorem. 
\begin{theorem} \label{thm_R_conv}
Let $t_0\in (0,1)$ and $k \in \{ 1, \ldots, M_b\} \cup \{ M - M_a + 1, \ldots, M\} $. Then, it holds
\begin{align*}
     (\bfR_{n,t} ( \lambda_{k,t} ) )_{t\in [t_0,1]} \rightsquigarrow (\bfR_t (\lambda_{k,t}) )_{t\in [t_0,1]} , 
 \end{align*} 
 the space $( \ell^\infty ( [t_0,1] )^{n_k^2}$, where $(\bfR_t (\lambda_{k,t}) )_{t\in [t_0,1]}$ is a centered Gaussian process.  For each $t\in [t_0,1],$ the random $n_k \times n_k$ matrix $ \bfR_t (\lambda_{k,t}) =(r_{ij;k}^{(t)})_{1 \leq i,j\leq n_k}$ is symmetric and has centered Gaussian entries with covariance function
\begin{align*}
    \cov \lb r_{ij;k}^{(s)}, r_{k'l;k}^{(t)} \rb 
    & = \omega_{s,t} \lb \E [ \xi(i) \xi(j) \xi(k') \xi(l) ]
    - \Sigma_{ij} \Sigma_{k'l} \rb 
    \\ & + (\theta_{s,t} - \omega_{s,t} ) \lb 
    \E [ \eta(i) \eta(k')] \E [ \eta (j) \eta (l)] 
    + \E [ \eta(i) \eta(l)] \E [ \eta (j) \eta (k')] 
    \rb 
\end{align*}
for $1 \leq i,j,k',l \leq n_k$, $s,t\in [t_0,1].$ Here, we define $\omega_{s,t} = \omega_{s,t,k,k}$
and $\theta_{s,t}= \theta_{s,t,k,k}.$
\end{theorem}
Proof of Theorem \ref{thm_R_conv} is postponed to Section \ref{sec_proof_thm_R_conv}.

The weak convergence of $(\bfR_{n,t} ( \lambda_{k,t} ) )_{t\in [t_0,1]}$ sets us in the position to prove our main result. 
\begin{proof} [Proof of Theorem \ref{thm_main}] 
To begin with, we argue that we may assume without loss of generality that the events
\begin{align*}
    \Omega_{n,t} = \{ \lambda_{n,j,t} \notin I_{t_0} \} \cap 
    \{ \lambda_1 (\bfS_{2,t}) , \ldots, \lambda_p (\bfS_{2,t}) \in I_{t_0} \}
\end{align*}
hold true for all $n \in \N, t\in [t_0,1]$ almost surely (with a null set not depending on $t$). To see this, define
\begin{align} \label{eq_support_lambda}
    \Omega_n = 
    \begin{cases}
     \{ \lambda_{n,j,t_0} > b_{y_{(t_0)} } \} 
    \cap \{ \lambda_1 ( \bfS_{2,1} ), \lambda_p(\bfS_{2,t_0} )  \in I_{t_0} \}
    \quad & \textnormal{ if } \alpha_k > 1 + \sqrt{y_{(t_0)}} , \\ 
    \{ \lambda_{n,j,1} < a_{y_{(t_0)} } \} 
    \cap \{ \lambda_1 ( \bfS_{2,1} ), \lambda_p(\bfS_{2,t_0} )  \in I_{t_0} \}
    \quad & \textnormal{ if } \alpha_k < 1 - \sqrt{y_{(t_0)}}.
    \end{cases} 
\end{align}
Note that $\Omega_n \subset \Omega_{n,t}$ for all $n\in\N, t\in [t_0,1]$ and $\PR (\Omega_n) \to 1.$   
Thus, considering the defintion of weak convergence in \cite[Definition 1.3.3]{vandervaart1996}, we may work without loss of generality under the assumption that $\Omega_{n,t}$ holds true for all $n\in\N, t\in[t_0,1]$ almost surely (with a nullset not depending on $t$).

Next, truncation of the random variables $x_{ij}$ under a finite fourth moment condition is a standard technical tool, see e.g. \cite{silverstein1989weak, bai2008limit} for the extreme eigenvalues of a sample covariance matrix. Thus, we may assume,
without loss of generality, that the random variables $x_{ij}$ are truncated at $\delta_n \sqrt{n}$, that is, $|x_{ij}|<\sqrt{n} \delta_n$, where $\delta_n\to0$ arbitrary slowly. This enables us to obtain appropriate concentration bounds on various terms appearing in the analysis. 

Moreover, we define
 \begin{align*}
     \hat \delta_{n,j,t}
     & = \sqrt{n} \lb \lambda_{n,j,t} - \lambda_{k,t} \rb, \\
     \bfK_{n,t}(\lambda) & = \bfS_{11,t} + \bfS_{12,t} \D_t \bfS_{21,t}, \\
     m_{1,t} (\lambda) & = \int \frac{x}{\lambda -x} d F_{y_{(t)}}(x).
 \end{align*}
 Note that 
 \begin{align}  
     \bfK_{n,t}( \lambda_{k,t}) & = \bfS_{1,t} + \bfX_{1,t} \bfA_{n,t}(\lambda_{k,t}) \bfX_{1,t}^\top 
     = \frac{1}{n} \xi_{1: \nt} ( \bfI + \bfA_{n,t}(\lambda_{k,t}) ) \xi_{1:\nt}^\top \nonumber \\
     & =  \frac{1}{\sqrt{n}} \bfR_{n,t}(\lambda_{k,t}) + \frac{1}{n} \bfSigma \tr \lb \bfI + \bfA_{n,t}(\lambda_{k,t}) \rb 
    \nonumber \\ &  = \frac{1}{\sqrt{n}} \bfR_{n,t}(\lambda_{k,t}) 
     + \lb  t + y m_{1,t} (\lambda_{k,t}) \rb \bfSigma + o_{\PR} \lb n^{-1/2} \rb \label{a4}
 \end{align}
 uniformly over $t\in [t_0,1]$. Here, we used the fact that uniformly for $t\in [t_0,1]$
\begin{align*}
    \frac{1}{n} \tr \lb \bfI + \bfA_{n,t} (\lambda_{k,t}) \rb 
    =   \frac{1}{n} \tr \lb \bfI_{\nt} + \bfA_{n,t} (\lambda_{k,t}) \rb 
    = t + y m_{1,t} (\lambda_{k,t}) + o_{\PR} \lb n^{-1/2} \rb,
\end{align*}
which follows from Lemma \ref{lem_tr_S_D_unif}. 
Note that 
\begin{align} \label{a5}
    \lambda_{n,j,t} \bfI - \bfK_{n,t} (\lambda_{n,j,t})
    & = 
    \lambda_{k,t} +  \frac{1}{\sqrt{n}} \hat\delta_{n,j,t} \bfI
    - \bfK_{n,t} (\lambda_{k,t}) 
    - \lb \bfK_{n,t} (\lambda_{k,t}) - \bfK_{n,t} (\lambda_{n,j,t}) \rb.
\end{align}
Furthermore, we have 
\begin{align}
    & \bfK_{n,t} (\lambda_{k,t}) - \bfK_{n,t} (\lambda_{n,j,t}) \nonumber \\
    & = \frac{1}{n} \xi_{1:\nt} \bfX_{2,t}^\top \left\{ 
    \lb \left[ \lambda_{k,t} + \frac{1}{\sqrt{n}} \hat\delta_{n,j,t} \right] \bfI - \bfS_{2,t} \rb \inv
    - \lb \lambda_{k,t} \bfI - \bfS_{2,t} \rb \inv 
    \right\} \bfX_{2,t} \xi_{1:\nt}^\top \nonumber \\
    & = - \frac{1}{\sqrt{n}} \hat\delta_{n,j,t} \frac{1}{n} \xi_{1:\nt} \bfX_{2,t}^\top  
    \lb \left[ \lambda_{k,t} + \frac{1}{\sqrt{n}} \hat\delta_{n,j,t} \right] \bfI - \bfS_{2,t} \rb \inv
     \lb \lambda_{k,t} \bfI - \bfS_{2,t} \rb \inv 
    \bfX_{2,t} \xi_{1:\nt}^\top \nonumber \\
     & = - \frac{1}{\sqrt{n}} \hat\delta_{n,j,t} 
     \lb y m_{3,t}(\lambda_{k,t}) \bfSigma + \op \rb 
     =  - \frac{1}{\sqrt{n}} \hat\delta_{n,j,t} 
      y m_{3,t}(\lambda_{k,t}) \bfSigma
      + o_{\PR}\lb n^{-1/2} \rb .  \label{a6}
\end{align}
Here, we used \eqref{eq_bound_D}, \cite[Lemma B.26]{bai2004}, and the $o_{\PR}$-terms do not depend on $t\in [t_0,1]$.
Combining \eqref{a4}, \eqref{a5} and \eqref{a6}, we have 
\begin{align}
   & \lambda_{n,j,t} \bfI - \bfK_{n,t} (\lambda_{n,j,t}) 
    \nonumber \\ & = \lambda_{k,t} \bfI 
    - \frac{1}{\sqrt{n}} \bfR_{n,t}(\lambda_{k,t}) 
     - \lb  t + y m_{1,t} (\lambda_{k,t}) \rb \bfSigma
     + \frac{1}{\sqrt{n}} \hat\delta_{n,j,t} \lb \bfI + 
      y m_{3,t}(\lambda_{k,t}) \bfSigma
      \rb + o_{\PR}\lb n^{-1/2} \rb.
      \label{a7}
\end{align}

Next, we apply a general almost sure representation theorem for the convergence $(\bfR_{n,t}) \rightsquigarrow (\bfR_t)_t$ given in Theorem \ref{thm_R_tight} and  \eqref{a7} (see \cite[Theorem 1.10.4]{vandervaart1996}) and we get almost sure represenations for these random variables in the sense of \cite{vandervaart1996}. As a consequence of \cite[Addendum 1.10.5]{vandervaart1996}, the almost sure representation of $(\bfR_t)_{t\in [t_0,1]}$ is measurable in $(\ell^\infty([t_0,1]))^M$ so that its distribution coincides with the distribution of the original process $(\bfR_t)_{t\in [t_0,1]}$ in the usual sense. 
For the sake of simplicity, we use the same notation for the almost sure representations as for the original random variables. Thus, we may assume that the convergence $(\bfR_{n,t}) \rightsquigarrow (\bfR_t)_t$ and \eqref{a7}
holds true almost surely (with a nullset not depending on $t$). 
Using these two convergence results and following the steps of \cite[Section 6.4]{bai2008central} closely, we see that $\hat\delta_{n,j,t}$ tends almost surely to a solution of
\begin{align*}
    \left| - \tilde{\bfR}_{t}^{(kk)} (\lambda_{k,t}) + 
    \lambda \lb 1 + y m_{3,t}(\lambda_{k,t}) \alpha_{k} \rb \bfI_{n_k}
    \right| = 0
\end{align*}
in $\lambda$, that is, an eigenvalue of
$$
\frac{1}{1 + y m_{3,t}(\lambda_{k,t}) \alpha_{k}}  \tilde{\bfR}_{t}^{(kk)} (\lambda_{k,t}) .
$$
Note that the corresponding nullset does not depend on $t\in [t_0,1]$. Consequently, we have
 \begin{align*}
        \left\{ \hat\delta_{n,j,t}
        \right\}_{j\in J_k, t\in [t_0,1]} \to
        \left\{ \lambda_j \lb \frac{1}{1 + y m_{3,t}(\lambda_{k,t}) \alpha_k} \tilde{\bfR}_{t}^{(kk)} (\lambda_{k,t}) \rb \right\}_{j\in J_k, t\in [t_0,1]}  
    \end{align*}
    almost surely, and thus, weakly. The weak convergence is also valid on the original probability space as a consequence of the almost sure representation theorem). 
\end{proof}

\subsection{Proof of Theorem \ref{thm_R_conv}} \label{sec_proof_thm_R_conv}
The objective of this section is to establish the weak convergence  of $(\bfR_{n,t}(\lambda_{k,t}))_{t\in [t_0,1]}$, as stated in Theorem \ref{thm_R_conv}.  
Using Theorem \ref{thm_functional_clt}, the assertion of Theorem \ref{thm_R_conv} follows from the two following results concerning the convergence of the finite-dimensional distributions of $(\bfR_{n,t}(\lambda_{k,t}))$ and its asymptotic tightness. 

To begin with, we provide the convergence of its finite-dimensional distributions. 
The following result will be proven in Section \ref{sec_proof_R_fidis}.
\begin{theorem} \label{thm_R_fidis}
Let $t_0\in (0,1)$ and $k \in \{ 1, \ldots, M_b\} \cup \{ M - M_a + 1, \ldots, M\} $.
The finite-dimensional distributions of $(\bfR_{n,t}(\lambda_{k,t } ))_{t\in [t_0,1]}$ converge weakly to the corresponding finite-dimensional distributions of $(\bfR_t(\lambda_{k,t}))_{t\in [t_0,1]}$ given in Theorem \ref{thm_R_conv}.
\end{theorem} 

 The following tightness result is proven in Section \ref{sec_proof_R_tight}. It relies on uniform bounds for the increments of the moments, which imply the asymptotic tightness of $(\bfR_{n,t}(\lambda_{k,t}))$.

\begin{theorem} \label{thm_R_tight}
Let $t_0\in (0,1)$ and $k \in \{ 1, \ldots, M_b\} \cup \{ M - M_a + 1, \ldots, M\} $.
Then, $( \bfR_{n,t} ( \lambda_{k,t}) )_{t\in[t_0,1]}$ is asymptotically tight in $( \ell^\infty ( [t_0,1] )^{M^2}.$ 
\end{theorem}

\subsubsection{Finite-dimensional distributions of $(\bfR_{n,t}(\lambda_{k,t}))$} \label{sec_proof_R_fidis}

In the following, we will often omit the parameter $\lambda_{k,t}$ in our notation for the sake of brevity. For example, we will write $\bfR_{n,t}$ and $\bfA_{n,t}$ instead of $\bfR_{n,t} (\lambda_{k,t})$ and $\bfA_{n,t}(\lambda_{k,t})$, respectively.

\begin{proof}[Proof of Theorem \ref{thm_R_fidis} ]
We apply Theorem \ref{thm_functional_clt} to the random variables
\begin{align} \label{def_U}
   U_{ijt} (\lambda_{k,t}) =  U_{ijt}= \frac{1}{\sqrt{n}} \lb  \mathbf{u}(i,t) \lb \bfI + \bfA_{n,t}(\lambda_{k,t}) \rb \mathbf{u}(j,t)^\top
   - \bfSigma_{ij} \tr \lb \bfI + \bfA_{n,t}(\lambda_{k,t}) \rb 
   \rb, 
    ~ 1 \leq i,j \leq M,  
\end{align}
for finitely many $t\in [t_0,1]$,
where we set
\begin{align*}
    \mathbf{u}(i,t) = \lb \xi_1(i), \ldots, \xi_{\nt}(i) \rb, ~ 1 \leq i \leq M. 
\end{align*}
In order to verify the conditions of this theorem, we need to show that the limits 
\begin{align}
    \omega_{s,t} & =\lim_{n\to\infty} \frac{1}{n} 
    \sum_{i=1}^{\ns \wedge \nt} 
    \lb \bfI + \bfA_{n,s} \rb_{ii}
     \lb \bfI + \bfA_{n,t} \rb_{ii},
    \label{omega_limit} \\
    \theta_{s,t} & = \lim_{n\to\infty} \frac{1}{n} 
    \sum_{i,j=1}^{\ns \wedge \nt} \lb \bfI + \bfA_{n,s} \rb_{ij}
     \lb \bfI + \bfA_{n,t} \rb_{ji}
    , \label{theta_limit}
\end{align}
exist (in probability). 
 
We claim that
\begin{align}
    & \omega_{s,t} = s \wedge t 
    -   (s \wedge t) y \left\{ \frac{1}{t} \lb 1 + \lambda_{k,t} m_{t} (\lambda_{k,t}) \rb 
    + \frac{1}{s} \lb 1 + \lambda_{k,s} m_{s} (\lambda_{k,s}) \rb 
    \right\} \nonumber \\ &
    + \frac{ ( s \wedge t ) y^2  \lb 1 + m_{1,s} (\lambda_{k,s}) \rb \lb 1 + m_{1,t} (\lambda_{k,t}) \rb  }{ \lb \lambda_{k,s} - y_{} \left[ 1 + m_{1,s}(\lambda_{k,s}) \right] \rb \lb \lambda_{k,t} - y_{} \left[ 1 + m_{1,t}(\lambda_{k,t}) \right] \rb }
    , \label{def_omega}
    \\ 
    & \theta_{s,t}  = s \wedge t 
     -   (s \wedge t) y \left\{ \frac{1}{t} \lb 1 + \lambda_{k,t} m_{t} (\lambda_{k,t}) \rb 
    + \frac{1}{s} \lb 1 + \lambda_{k,s} m_{s} (\lambda_{k,s}) \rb 
    \right\} \nonumber 
     \\ & + 
     s \wedge t \Bigg\{ y^2 m_s(\lambda_{k,s}) m_t (\lambda_{k,t}) + 
             \frac{ (s \wedge t ) y m_s(\lambda_{k,s}) m_t(\lambda_{k,t})}{1 - s \wedge t y \lambda_{k,t} \lambda_{k,s} \su_s (\lambda_{k,s}) \su_t (\lambda_{k,t}) m_s (\lambda_{k,s}) m_t (\lambda_{k,t}) } \nonumber \\ & \times 
            \lb 1 
            + y \lambda_{k,s} \su_s (\lambda_{k,s}) m_s (\lambda_{k,s}) 
            + y \lambda_{k,t} \su_t (\lambda_{k,t})  m_t (\lambda_{k,t})
            + y^2 \lambda_{k,s} \lambda_{k,t} \su_s (\lambda_{k,s}) \su_t (\lambda_{k,t}) m_s (\lambda_{k,s}) m_t (\lambda_{k,t})
            \rb
            \Bigg\} 
            \nonumber \\ & \quad \times  
             \lb 1 
            + y \lambda_{k,s} \su_s (\lambda_{k,s}) m_s (\lambda_{k,s})
            + y \lambda_{k,t} \su_t (\lambda_{k,t})  m_t (\lambda_{k,t})
            + y^2 \lambda_{k,s} \lambda_{k,t} \su_s (\lambda_{k,s}) \su_t (\lambda_{k,t}) m_s (\lambda_{k,s}) m_t (\lambda_{k,t})
            \rb.
    \label{def_theta}  
\end{align} 
To see this, we multiply all terms out and use from Lemma \ref{lem_tr_S_D} that, if $s= s\wedge t$,
\begin{align*}
     \frac{1}{n} \sum\limits_{i=1}^{\ns \wedge \nt } \lb \bfA_{n,t} \rb_{ii}
     & = \frac{1}{n} \sum\limits_{i=1}^{\ns} \lb \bfA_{n,t} \rb_{ii}
       = \frac{1}{n} \tr \bfS_{2,s} \lb \lambda_{k,t}\bfI - \bfS_{2,t}\rb\inv 
     = -  \frac{s}{t} y \lb 1 + \lambda_{k,t} m_t (\lambda_{k,t}) \rb + \op, \\
     \frac{1}{n} \sum\limits_{i=1}^{\ns \wedge \nt } \lb \bfA_{n,s} \rb_{ii} 
     & = - y \lb 1 + \lambda_{k,s} m_s (\lambda_{k,s}) \rb + \op,
    \end{align*}
and similarly to \cite[Lemma 6.1]{bai2008central},
\begin{align*}
    \frac{1}{n} 
    \sum_{i=1}^{\ns \wedge \nt} 
    \lb  \bfA_{n,s} \rb_{ii}
     \lb  \bfA_{n,t} \rb_{ii}
     \conp
     \frac{ ( s \wedge t ) y^2  \lb 1 + m_{1,s} (\lambda_{k,s}) \rb \lb 1 + m_{1,t} (\lambda_{k,t}) \rb  }{ \lb \lambda_{k,s} - y_{} \left[ 1 + m_{1,s}(\lambda_{k,s}) \right] \rb \lb \lambda_{k,t} - y_{} \left[ 1 + m_{1,t}(\lambda_{k,t}) \right] \rb }.
\end{align*}
For \eqref{def_theta}, we additionally used 
\begin{align*}
    & \frac{1}{n} \sum\limits_{i=1}^{\ns} \lb \bfA_{n,s} \rb_{ij} \lb \bfA_{n,t} \rb_{ji}
     = \frac{1}{n} 
    \tr \bfS_{2,s} \lb \lambda_{k,s}\bfI - \bfS_{2,s} \rb\inv \bfS_{2,s} \lb \lambda_{k,t}\bfI - \bfS_{2,t} \rb\inv, 
\end{align*}
and combined it with Lemma \ref{lem_tr_S_D_S_D}. 
It is a routine to verify that $1+m_{1,t}(\lambda) = - \lambda m_{1,t}(\lambda)$. Applying this formula to \eqref{def_omega} gives the claimed form of $\omega_{s,t}$, and the proof concludes. 
\end{proof}

\subsubsection{Asymptotic tightness of $( \bfR_{n,t}(\lambda_{k,t}) )$} \label{sec_proof_R_tight}

For the following discussion, we need to introduce some notations. 
Let $s,t\in (0,1]$ with $s< t$ and $1\leq i \neq j  \leq \nt$. Then, we define
\begin{align*}
    \D_t(\lambda_{k,t}) & = \D_t =  (  \bfS_{2,t} - \lambda_{k,t}\bfI)\inv, \\
    \bfX_{2,(s,t]} &= 1/\sqrt{n} (\eta_{ \ns + 1}, \ldots, \eta_{\nt}) \in \R^{p \times (\nt - \ns )} , \\
    \bfS_{2,(s,t]} & = \bfX_{2,(s,t]} \bfX_{2,(s,t]}^\top, \\
    \D_{i,t} (\lambda_{k,t}) & = \D_{i,t}
    = \lb \bfS_{2,t} - \frac{1}{n} \eta_i \eta_i^\top - \lambda_{k,t}\bfI \rb\inv, \\
     \D_{i,j,t} (\lambda_{k,t}) & = \D_{i,j,t}
    = \lb \bfS_{2,t} - \frac{1}{n} \eta_i \eta_i^\top - \frac{1}{n} \eta_j \eta_j^\top - \lambda_{k,t}\bfI \rb\inv, \\
    \beta_{i,t}(\lambda_{k,t}) &= \beta_{i,t}
    = \frac{1}{1 + n\inv \eta_i^\top \D_{i,t} \eta_i }. \\
\end{align*}

\begin{proof}[Proof of Theorem \ref{thm_R_tight}]
Recall that the random variables $x_{ij}$ are truncated at $\sqrt{n}\delta_n$, which will be used throughout this proof when applying concentration inequalities for quadratic forms such as \cite[Lemma B.26]{bai2004}. Moreover, recall that we always work under the event $\Omega_n$ in \eqref{eq_support_lambda}, which guarantees that the spectral norm $\| \D_t \|$ of $\D_t$ and similar resolvent-type matrices is bounded uniformly over the parameters $n,t$. 

In the following, we will show that $U_{ijt}  - U_{ijs} $ satisfies the moment condition \eqref{mom_cond_tightness2}, where the sesquilinear form $U_{ijt}$ is defined in \eqref{def_U}. Without loss of generality we assume $s=s \wedge t$ and let $\bfB_{n,t}= (b_{ij})_{1 \leq i,j \leq \nt} = \bfI + \bfA_{n,t}.$
To begin with, we decompose the difference $U_{ijt} - U_{ijs}$ in two parts, namely,
\begin{align*}
    U_{ijt} - U_{ijs}
    = Z_{1} + Z_{2},
\end{align*}
where
\begin{align*}
    Z_1   & =  \frac{1}{\sqrt{n}} \lb 
    \mathbf{u}(i,s) \left\{ \lb b_{ij}\rb_{1 \leq i,j \leq \ns} - \bfB_{n,s} \right\} \mathbf{u}(j,s)^\top
   - \bfSigma_{ij} \tr \left\{ \lb b_{ij}\rb_{1 \leq i,j \leq \ns} - \bfB_{n,s} \right\}
    \rb , \\
    Z_2  & =  \frac{1}{\sqrt{n}} \lb 
    \mathbf{u}(i,t) \left\{ \lb b_{ij}\rb_{\ns + 1 \leq i,j \leq \nt} \right\} \mathbf{u}(j,t)^\top
   - \bfSigma_{ij} \tr \left\{ \lb b_{ij}\rb_{\ns + 1 \leq i,j \leq \nt}  \right\}
    \rb.
\end{align*}
We aim to show that $Z_1$ and $Z_2$ satisfy the moment condition \eqref{mom_cond_tightness2} from Theorem \ref{thm_functional_clt} with $b'=2, b=4$. Let us start with the term $Z_2.$

\paragraph*{Analysis of $Z_2$}
Using $2 \bfx^\top \bfA \bfy = \lb \bfx + \bfy \rb^\top \bfA \lb \bfx + \bfy \rb  
- \bfx^\top \bfA \bfx - \bfy^\top \bfA \bfy$,  and \cite[Lemma B.26]{bai2004}, we see that 
\begin{align*}
    \E [ Z_2 ^ 4 ] & \lesssim \frac{1}{n^2}
   \left\{   \tr \lb  \lb b_{ij}\rb_{\ns + 1 \leq i,j \leq \nt} \rb^2
   \right\}^2 \\
   & \lesssim \frac{\lb \nt - \ns \rb ^2  }{n^2}  
   + \frac{1}{n^2} \lb \tr  \bfS_{2,(s,t]} \D_t \bfS_{2,(s,t]} \D_t \rb ^2 \\
   &=  \frac{\lb \nt - \ns \rb ^2  }{n^2}  
   + Z_{2,1} + Z_{2,2},
\end{align*}
where
\begin{align*}
    Z_{2,1} &= \frac{1}{n^2} \lb  \sum\limits_{i=\ns + 1}^{\nt} \lb \frac{1}{n} \eta_i^\top \D_t \eta_i \rb^2 \rb^2 , \\ 
    Z_{2,2} &= \frac{1}{n^2} \lb  \sum\limits_{\substack{i,k=\ns + 1, \\ i \neq k }}^{\nt} \frac{1}{n^2} \eta_i^\top \D_t \eta_k \eta_k^\top \D_t \eta_i \rb ^2.
\end{align*}
Investigating $Z_{2,1}$, we decompose it further as
\begin{align*}
    Z_{2,1} & = Z_{2,1,1}+ Z_{2,1,2},
\end{align*}
where
\begin{align*}
    Z_{2,1,1} &= \frac{1}{n^2} \sum\limits_{i=\ns +1}^{\nt}\lb \frac{1}{n} \eta_i^\top \D_t \eta_i \rb^4, \\
    Z_{2,1,2} & = \frac{1}{n^2} \sum\limits_{\substack{i,k=\ns +1, \\ i \neq k}}^{\nt} \lb \frac{1}{n} \eta_i^\top \D_t \eta_i \rb^2 \lb \frac{1}{n} \eta_k^\top \D_t \eta_k \rb^2. 
\end{align*}
Using \eqref{eq_sher_mor}, we get for the first summand in the decomposition of $Z_{2,1}$
    \begin{align*}
        Z_{2,1,1} & \lesssim Z_{2,1,1,1} + Z_{2,1,1,2} ,
    \end{align*}
    where
    \begin{align*}
        Z_{2,1,1,1} & = \frac{1}{n^2} \sum\limits_{i=\ns + 1}^{\nt}  \lb \frac{1}{n} \eta_i^\top \D_{i,t} \eta_i \rb^4  , \\ 
        Z_{2,1,1,2} & = \frac{1}{n^2} \sum\limits_{i=\ns + 1}^{\nt} \lb \beta_{i,t} \rb^4 \lb \frac{1}{n} \eta_i^\top \D_{i,t} \eta_i \rb^8 . 
    \end{align*}
    Note that for $q\in2\N$
    \begin{align*}
        \lb  \frac{1}{n} \eta_i^\top \D_{i,t} \eta_i \rb^q 
        \lesssim  \lb  \frac{1}{n} \eta_i^\top \D_{i,t} \eta_i -  n \inv \tr \D_{i,t}\rb^q 
        + \lb n \inv \tr \D_{i,t} \rb^q 
        \lesssim \lb  \frac{1}{n} \eta_i^\top \D_{i,t} \eta_i -  n \inv \tr \D_{i,t}\rb^q + 1.
    \end{align*}
    Using \cite[Lemma B.26]{bai2004}, \eqref{eq_bound_D} and \eqref{eq_bound_beta}, we see that
    \begin{align*}
        \E | Z_{2,1,1,1} | 
        \lesssim \frac{\nt - \ns}{n^2}, \quad 
        \E | Z_{2,1,1,1} | 
        \lesssim \frac{\nt - \ns}{n^2}, 
    \end{align*}
    and thus,
        \begin{align*}
        \E | Z_{2,1,1} | 
        \lesssim \frac{\nt - \ns}{n^2}. 
    \end{align*}
Using H\"older's inequality, the term $Z_{2,1,2}$ can be handled similarly to $Z_{2,1,1}$ and thus, we get
    \begin{align*}
        \E | Z_{2,1,2} | 
        \lesssim \lb \frac{\nt - \ns}{n} \rb^2 . 
    \end{align*}
    Consequently, we have the desired inequality for $Z_{2,1}$, namely
    \begin{align*}
        \E | Z_{2,1}|  \lesssim \frac{\nt - \ns}{n^2} + \lb \frac{\nt - \ns}{n} \rb^2,
    \end
    {align*}
    and it is left to investigate $Z_{2,2}.$
    We decompose $Z_{2,2} = Z_{2,2,1} + Z_{2,2,2},$ where
    \begin{align*}
        Z_{2,2,1} & = \frac{1}{n^2}  \sum\limits_{\substack{i,k,k'=\ns + 1, \\ i \notin \{k, k'\} }}^{\nt} \frac{1}{n^4} \eta_i^\top \D_t \eta_k \eta_k^\top \D_t \eta_i 
        \eta_i^\top \D_t \eta_{k'} \eta_{k'}^\top \D_t \eta_i , \\ 
        Z_{2,2,2} & = 
        \frac{1}{n^2}  \sum\limits_{\substack{i,i',k,k'=\ns + 1, \\ i \neq k, \\ i' \neq k', \\ i \neq i' }}^{\nt} \frac{1}{n^4} \eta_i^\top \D_t \eta_k \eta_k^\top \D_t \eta_i 
        \eta_{i'}^\top \D_t \eta_{k'} \eta_{k'}^\top \D_t \eta_{i'} .
    \end{align*}
Using the same techniques as above, especially \eqref{eq_sher_mor}, and \cite[Lemma B.26]{bai2004}, we get
\begin{align*}
  &  \E |  \frac{1}{n^4} \eta_i^\top \D_t \eta_k \eta_k^\top \D_t \eta_i 
        \eta_{i'}^\top \D_t \eta_{k'} \eta_{k'}^\top \D_t \eta_{i'} | 
        \lesssim n^{-2} \lb \E || \D_{i,k,t} ||^q \rb\sq
        \lesssim n^{-2}
\end{align*}
for some $q\geq 2.$
This shows that 
\begin{align*}
    \E | Z_{2,2,1} | & \lesssim \frac{\nt - \ns }{n^2}, \quad 
    \E | Z_{2,2,2} |  \lesssim \lb \frac{\nt - \ns }{n} \rb^2.
\end{align*}

\paragraph*{Analysis of $Z_1$}  
Similarly to the analysis of $Z_2$, an application of \cite[Lemma B.26]{bai2004} gives 
\begin{align*}
   &  \E [ Z_1 ^2] \lesssim 
    \frac{1}{n}  \tr \lb \bfS_{2,s} (\D_t - \D_s ) \rb^2  
    = \frac{1}{n}  \tr \lb \bfS_{2,s} \D_t \bfS_{2,(s,t]} \D_s \rb^2  
   \\ &  = \frac{1}{n} \sum\limits_{j,k = \ns + 1}^{\nt} \frac{1}{n^2} \eta_j \D_s \bfS_{2,s} \D_t \eta_k \eta_k^\top \D_s \bfS_{2,s} \D_t \eta_j 
    = Z_{1,1} + Z_{1,2},
\end{align*}
where 
\begin{align*}
    Z_{1,1} & = \frac{1}{n} \sum\limits_{j = \ns + 1}^{\nt} \lb \frac{1}{n} \eta_j \D_s \bfS_{2,s} \D_t \eta_j \rb^2  , \\
    Z_{1,2} & = \frac{1}{n} \sum\limits_{\substack{j,k = \ns + 1, \\ j\neq k}}^{\nt} \frac{1}{n^2} \eta_j \D_s \bfS_{2,s} \D_t \eta_k \eta_k^\top \D_s \bfS_{2,s} \D_t \eta_j 
\end{align*}
First, we investigate the first term $Z_{11}.$ Note that both $\D_s$ and $\bfS_{2,s}$ are independent of $\eta_j$ for $\ns +1 \leq j \leq \nt.$ Thus, we use \eqref{eq_sher_mor} to replace $\D_t$ and get
\begin{align*}
    Z_{1,1} & \lesssim \frac{1}{n} \sum\limits_{j = \ns + 1}^{\nt} \left\{ \lb \frac{1}{n} \eta_j \D_s \bfS_{2,s} \D_{i,t} \eta_j \rb^2 
    + \beta_{i,t}^2  \lb \frac{1}{n} \eta_j \D_s \bfS_{2,s} \D_{i,t} \eta_j \rb^4 \right\} \\
    & \lesssim \frac{1}{n} \sum\limits_{j = \ns + 1}^{\nt} \Big\{ \lb \frac{1}{n} \eta_j \D_s \bfS_{2,s} \D_{i,t} \eta_j - \frac{1}{n} \tr \D_s \bfS_{2,s} \D_{i,t} \rb^2 \\ & 
    + \beta_{i,t}^2  \lb \frac{1}{n} \eta_j \D_s \bfS_{2,s} \D_{i,t} \eta_j 
     - \frac{1}{n} \tr \D_s \bfS_{2,s} \D_{i,t} \rb^4 
      +  \lb \frac{1}{n} \tr \D_s \bfS_{2,s} \D_{i,t}\rb^2 
     + \lb \frac{1}{n} \tr \D_s \bfS_{2,s} \D_{i,t}\rb^4 \Big\} .
\end{align*}
Using \cite[Lemma B.26]{bai2004} and further decompositions via \eqref{eq_sher_mor}, we have
\begin{align*}
    \E [ Z_{1,1}] \lesssim \frac{\nt - \ns}{n^2}. 
\end{align*}
The second term can be treated similarly, and we get
    \begin{align*}
        \E [ Z_{1,2}] \lesssim  \lb \frac{\nt - \ns}{n} \rb^2, 
    \end{align*}
    which implies 
    \begin{align*}
        \E [ Z_1] \lesssim \lb \frac{\nt - \ns}{n} \rb^2 + \frac{\nt - \ns}{n^2}.  
    \end{align*}
\end{proof}

\section{Auxiliary Results}

This section is devoted to the presentation of additional results that facilitate the proof and application of our main theorems. All their proofs are deferred to Appendix \ref{sec_proof_aux}. 
In the following, we provide useful expressions for terms appearing in the covariance structure of the process of spiked eigenvalues, which enables us to derive explicit expressions for $\omega_{s,t}$ and $\theta_{s,t}$ in \eqref{def_omega} and \eqref{def_theta}, respectively.   These formulas can be derived by a direct calculation using the density of the \MP distribution. 
\begin{lemma} \label{lem_formula_m_lambda}
    \begin{align*}
        m_{1,t}(\lambda_{k,t}) &= 
        \frac{t}{\alpha_{k} -t }, \\
        m_{3,t}(\lambda_{k,t}) &= 
        \frac{t}{\lb \alpha_{k} - t \rb^2 - t y}, \\ 
    m_t(\lambda_{k,t}) & = \frac{ 1 - \alpha_k}{(\alpha_k -t) \lb  (\alpha_k - 1)t + y\rb}, \\ 
        \su_t (\lambda_{k,t}) &= 
        \frac{ (1 - \alpha_k ) (\alpha_k - t + y)}{ \alpha_k (\alpha_k - t) ((\alpha_k - 1) t + y)}.
         \end{align*} 
\end{lemma}
     
    The following two lemmas are crucial for calculating the covariance structure of $(\bfR_{n,t}).$ To begin with, the following result is needed to determine $\omega_{s,t}$ in \eqref{def_omega}. 
    
\begin{lemma} \label{lem_tr_S_D}
For $s = s \wedge t, $ we have
    \begin{align*}
       - \frac{1}{n} \tr \bfS_{2,s} \lb \lambda_{k,t}\bfI - \bfS_{2,t}\rb\inv 
       = \frac{1}{n} \tr \bfS_{2,s} \D_t 
        = \frac{s}{t} y \lb 1 + \lambda_{k,t} m_t \rb + o_{\PR}(1). 
    \end{align*}
\end{lemma}

The statement of Lemma \ref{lem_tr_S_D} can be strengthened such that the negligible term does not depend on the sequential parameters $s,t$, which is needed in the proof of Theorem \ref{thm_main}. For convenience, we restrict ourselves to the case $s=t$ for the following lemma.

\begin{lemma}
    \label{lem_tr_S_D_unif}
    It holds that
    \begin{align*}
        \frac{1}{n} \tr \bfS_{2,t} \D_t 
        = y ( 1 + \lambda_{k,t} m_t) + o_{\PR} \lb n^{-1/2} \rb 
        = - y m_{1,t} + o_{\PR} \lb n^{-1/2} \rb ,
    \end{align*}
    where the $o_{\PR} \lb n^{-1/2} \rb $-term does not depend on $t\in [t_0,1].$ 
\end{lemma}

The next results facilitate the calculation of $\theta_{s,t}$ in \eqref{def_theta}.
\begin{lemma}
    \label{lem_tr_S_D_S_D}
    For $s = s \wedge t, $ we have
    \begin{align*}
       &  \frac{1}{n} 
    \tr \bfS_{2,s} \lb \lambda_{k,s}\bfI - \bfS_{2,s} \rb\inv \bfS_{2,s} \lb \lambda_{k,t}\bfI - \bfS_{2,t} \rb\inv \\
    & = \frac{1}{n}   \tr \bfS_{2,s} \D_s \bfS_{2,s}\D_t \\
    & = s  \Bigg\{ y^2 m_s m_t + 
             \frac{s y m_s m_t}{1 - s y \lambda_{k,t} \lambda_{k,s} \su_s \su_t m_s m_t }   \\
    & \quad \times \left.         
            \left( 1 
            + y \lambda_{k,s} \su_s m_s 
            + y \lambda_{k,t} \su_t  m_t 
            + y^2 \lambda_{k,s} \lambda_{k,t} \su_s \su_t m_s m_t
            \right)
            \right) \Bigg\} \\ 
    & \quad \times  
             \lb 1 
            + y \lambda_{k,s} \su_s m_s 
            + y \lambda_{k,t} \su_t  m_t 
            + y^2 \lambda_{k,s} \lambda_{k,t} \su_s \su_t m_s m_t
            \rb + \op. 
    \end{align*}
\end{lemma}

The following auxiliary result is needed for the proof of Lemma \ref{lem_tr_S_D_S_D}. 
\begin{lemma} \label{lem_conv_tr_DsDt}
It holds for $s = s \wedge t$ that
   \begin{align} \label{eq_conv_tr_DsDt}
        \frac{1}{n} \tr \D_{s} \D_t  
        & = 
        \frac{y m_s m_t}{1 - s y \lambda_{k,t} \lambda_{k,s} \su_s \su_t m_s m_t }
        + o_{\PR}(1).
    \end{align}

\end{lemma}

\begin{appendix}
\section{A functional CLT for sesquilinear forms}
The following result is crucial for investigating the asymptotic behavior of $(\bfR_{n,t})$. Since it might be of independent interest, we state it in a general form. For this purpose, we consider a sequence $(\bfx_n,\bfy_n)_{n\in\N}$ of i.i.d. centered random vectors in $\R^k\times \R^k$, where $k\in\N$ is a fixed parameter. Furthermore, we denote for $n\in\N$, $t\in [0,1]$ and $1 \leq l \leq k$
\begin{align*}
    \bfx_n = \lb x_{1i},\ldots, x_{ki}\rb^\top \in \R^k, ~ 
    \bfx(l,t) = \lb x_{l1}, \ldots, x_{l\nt}\rb^\top \in \R^{\nt}. 
\end{align*}
Similarly, we define $\bfy_1, \ldots, \bfy_n$ and $\bfy(1,t), \ldots, \bfy(k,t).$ Let $\rho(l) = \E [ x_{l1} y_{l1}]$ for $1 \leq l \leq k.$
\begin{theorem} \label{thm_functional_clt}
For each $t\in [0,1]$ and $l\in \{1, \ldots, k\}$ let $(\bfB_{n,t}(l))_{n\in\N}$ be a sequence of $\nt \times \nt$ symmetric matrices.
\begin{enumerate}
    \item  We assume that for all $s,t\in [0,1], 1 \leq l, m \leq k$, the following limits exist
\begin{align}
    \omega_{s,t,l,m} & =\lim_{n\to\infty} \frac{1}{n} 
    \sum_{i=1}^{\ns \wedge \nt} 
    \lb \bfB_{n,s}(l) \rb_{ii}
     \lb \bfB_{n,t}(m) \rb_{ii}, \label{omega}
    \\
    \theta_{s,t,l,m} & = \lim_{n\to\infty} \frac{1}{n} 
    \sum_{i,j=1}^{\ns \wedge \nt} \lb \bfB_{n,s} (l) \rb_{ij}
     \lb \bfB_{n,t} (m) \rb_{ji} \label{theta}
    .
\end{align}
Then, the finite-dimensional distributions of 
\begin{align*}
    U_{l,t}
    = \frac{1}{\sqrt{n}} \lb \bfx(l,t)^\top \bfB_{n,t}(l) \bfy(l,t)
    - \rho (l) \tr \bfB_{n,t}(l) \rb , ~ t\in [0,1], ~ 1 \leq l \leq k,
\end{align*}
converge to the finite-dimensional distributions of $(W_{l,t})_{ t\in [0,1], 1 \leq l \leq k}$, whose marginals are zero-mean Gaussians with covariance structure 
\begin{align*}
    \cov (U_{l,s}, U_{m,t} ) 
    = \omega_{s,t,l,m} M_{1,l,m}
    +  \lb \theta_{s,t,l,m} - \omega_{s,t,l,m} \rb \lb  M_{2,l,m} + M_{3,l,m} \rb 
\end{align*}
for $s,t\in [0,1], 1 \leq l, m \leq k$, where
\begin{align*}
    M_{1,l,m}  & = \E [ x_{l1} y_{l1} x_{m1} y_{m1} ] - \rho(l) \rho(m) , \\
    M_{2,s,t,l,m} & =   \E [ x_{l1}  x_{m1} ]
    \E[ y_{l1} y_{m1} ],
    \\ 
    M_{3,l,m} & = \E [ x_{l1} y_{m1} ]
    \E[ x_{m1} y_{l1} ]. 
\end{align*}
\item  Let $t_0 \in (0,1).$ We assume that the limits in \eqref{omega} and \eqref{theta} exist for all $s,t\in [t_0,1].$ Moreover, if 
\begin{align}  \label{mom_cond_tightness}
      \E |  U_{l,t} - U_{l,s}   |^b 
      \leq K \left\{  \lb \frac{\nt - \ns }{n} \rb^{1+\delta} +  \frac{\nt - \ns }{n^{1+\delta'}} \right\} 
\end{align}
for some $b, \delta, \delta' >0$, all $t_0 \leq s \leq t \leq 1$ and some constant $K>0$ independent of $n,s,t$,
then $(U_{l,t} )_{t\in[t_0,1]}$ is asymptotically tight, and we have for $n\to\infty$
\begin{align*}
    (U_{l,t} )_{t\in[t_0,1], 1\leq l \leq k} 
    \rightsquigarrow (W_{l,t} )_{t\in[t_0,1], 1\leq l \leq k}
    \textnormal{ in } \lb \ell^\infty ([t_0,1]) \rb^k .
\end{align*}
\end{enumerate}
\end{theorem}
\begin{proof}[Proof of Theorem \ref{thm_functional_clt}]
   The first part is a slight generalization of \cite[Corollary 2.2]{qinwen2014joint}. The crucial point in our model, compared to the previous work \cite{qinwen2014joint} is that the definition of the matrices $\bfB_{n,t}(l)$ and of the random vectors $\bfx(l,t)$, $\bfy(l,t)$ depends on the sequential parameter $t$ and their dimensions vary with $t$. However, the proof is very similar to \cite{qinwen2014joint}. For the sake of brevity, we omit the details. 
   For the second part, we note that it suffices to show that $(U_{l,t} )_{t\in[t_0,1], 1\leq l \leq k}$ is asymptotically tight in the space $\ell^\infty([t_0,1])$ (see \cite[Lemma 1.4.3, Theorem 1.5.4]{vandervaart1996}). 
   For this purpose, we show that the conditions of \cite[Corollary A.4]{dette2019determinants} are satisfied.
   Define for $t_0 \leq r \leq s \leq t $
   \begin{align*}
       m(r,s,t) = \min \left\{ | U_{l,t} - U_{l,s} | , | U_{l,r} - U_{l,s} | 
       \right\} .
   \end{align*}
   Then, we have for $\eta > 0$, 
   \begin{align*}
       & \PR ( m(r,s,t) > \eta ) \leq 
       \PR (| U_{l,t} - U_{l,s} | > \eta ) + \PR (| U_{l,r} - U_{l,s} |  > \eta )
       \\ & \leq \frac{1}{\eta^b} \E | U_{l,t} - U_{l,s} |^b + \frac{1}{\eta^b} \E | U_{l,r} - U_{l,s} |^b 
   \end{align*}
   In the case $t -r \geq 1/n$, we further estimate using \eqref{mom_cond_tightness} and the basic inequality $x -1 \leq \lfloor x \rfloor \leq x +1$ for all $x\in \R$
   \begin{align*}
       \E | U_{l,t} - U_{l,s} |^b \lesssim \lb \frac{\nt - \ns }{n} \rb^{1+\delta} +  \frac{\nt - \ns }{n^{1+\delta'}}
       \lesssim \lb t -r \rb ^{1+\varepsilon} 
   \end{align*}
   where $\varepsilon = \delta \vee \delta',  $ and similarly,
   \begin{align*}
       \E | U_{l,t} - U_{l,s} |^b \lesssim \lb t -r \rb ^{1+\varepsilon}.
   \end{align*}
   In the other case $t -r < 1/n$, we have $\lfloor nr \rfloor = \ns$ or $\ns = \nt$, and thus, $m(r,s,t)=0$ in this case. In summary, we obtain the inequality
   \begin{align*}
       \PR ( m(r,s,t) > \eta) \lesssim \frac{1}{\eta^b} \lb t -r \rb ^{1+\varepsilon}.
   \end{align*}
   Moreover, we have
   \begin{align*}
       \PR ( | U_{l,t} - U_{l,s} | > \eta) 
       \leq \frac{1}{\eta^b} \E | U_{l,t} - U_{l,s} |^b
       \lesssim \frac{1}{\eta^b} \left\{  \lb t - s + \frac{1}{n} \rb^{1+\delta} + \frac{ t- s + \frac{1}{n} }{n^{\delta'}}  \right\} .
   \end{align*}
Let $m\in\N$ and 
\begin{align*}
    K_j = \left[ \frac{j - 1}{m} , \frac{j}{m}  \right], \quad
    \lfloor m t_0 \rfloor \leq j \leq m. 
\end{align*}
Combining the inequalities above, we are able to apply \cite[Corollary A.4]{dette2019determinants}, which gives
\begin{align*}
    \PR \lb \sup_{s,t \in K_j} | U_{l,t} - U_{l,s} | > \eta \rb 
    \lesssim \frac{1}{m^{1+\varepsilon}} +  \lb \frac{1}{m} + \frac{1}{n} \rb^{1+\delta} + \frac{ \frac{1}{m}+ \frac{1}{n} }{n^{\delta'}}.
\end{align*}
This implies
\begin{align*}
    \limsup_{n\to\infty} \PR \lb \sup_{ \lfloor mt_0 \rfloor \leq j \leq m} \sup_{s,t \in K_j} | U_{l,t} - U_{l,s} | > \eta \rb 
    \lesssim \frac{1}{m^{\varepsilon}} \to 0 , \textnormal{ as } m\to\infty.
\end{align*}
Since the finite-dimensional distributions of $(U_{l,t})$ converge weakly by assumption, \cite[Theorem 1.5.6]{vandervaart1996} yields the asymptotic tightness of $(U_{l,t}).$ 
\end{proof}

\begin{remark}
The proof of Theorem \ref{thm_functional_clt} shows that the moment condition \eqref{mom_cond_tightness} can be replaced by the following slightly more general condition: If the increments can be decomposed as $U_{l,t} - U_{l,s} = Z_{1,l,s,t} + Z_{2,l,s,t}$, then \eqref{mom_cond_tightness} can be replaced by  
\begin{align} \label{mom_cond_tightness2}
    \E | Z_{1,l,s,t} |^b \vee \E | Z_{2,l,s,t} |^{b'} & \leq K \left\{  \lb \frac{\nt - \ns }{n} \rb^{1+\delta} +  \frac{\nt - \ns }{n^{1+\delta'}} \right\}
\end{align}
for some $b,b',\delta, \delta'>0. $
\end{remark} 

\section{Proofs of Auxiliary Results} \label{sec_proof_aux}

\begin{proof}
    [Proof of Lemma \ref{lem_formula_m_lambda}]
    From \cite[Lemma B.1]{paul2007asymptotics}, we have 
    \begin{align*}
        \frac{t}{\alpha_{k} - t} 
        = \frac{1}{ \frac{1}{t} \alpha_{k} - 1}
        = \int \frac{x}{ \frac{1}{t} \alpha_{k} - x} d \tilde F_{y_{(t)}} (x)
        =  \int \frac{x}{ \alpha_{k} - x} d  F_{y_{(t)}} (x)
        = m_{1,t}(\lambda_{k,t}).
    \end{align*}
    The expression for $m_{3,t}(\lambda_{k,t})$ follows similarly from \cite[Lemma B.2]{paul2007asymptotics}. 
    Using \eqref{eq_m_mu} and
    \begin{align*}
        m_{1,t}(\lambda_{k,t} ) = 
        - 1 - \lambda_{k,t} m_t(\lambda_{k,t}), 
    \end{align*}
    the proof of Lemma \ref{lem_formula_m_lambda} concludes. 
\end{proof}
We continue with preparations for the proofs of the remaining lemmas. For this purpose, let $\su_{n,t}(\lambda_{k,t}) = \su_{n,t}$ denote the Stieltjes transform of the companion matrix $\underline{\bfS}_{2,t} = \bfX_{2,t}^\top \bfX_{2,t} \in \R^{\nt \times \nt}$. Then, we have the following auxiliary result. 
    \begin{lemma} \label{lem_conv_stieltjes}
    For $t\in [t_0,1]$, we have
    $$
    \E [ \su_{n,t}(\lambda_{k,t})] \to \su_t (\lambda_{k,t}), \quad n\to\infty. 
    $$
    \end{lemma}
    From the standard outward appearance of the convergence provided in Lemma \ref{lem_conv_stieltjes}, its result may seem to follow immediately from well-known results on the convergence of the Stieltjes transform such as \cite{baizhou2008}.
    However, it is important to point out that these results typically guarantee the convergence of the expected Stieltjes transform only for arguments in the upper half plane $\mathbb{C}^+$, which are naturally well separated from the spectrum of $\underline \bfS_{2,t}$. 
    Thus, an additional argument is needed to extend  rigorously to $\E [\su_{n,t} (\lambda_{k,t})]$ to control for outliers in the spectrum of the sample covariance matrix that may potentially cause a singularity. 
    
\begin{proof}[Proof of Lemma \ref{lem_conv_stieltjes}]
        For all $z\in\mathbb{C}^+,$ we have from \cite{baizhou2008},
        \begin{align*}
            \E [ \su_{n,t} (z ) ] \to \su_t (z) = o(1), \quad n\to\infty. 
        \end{align*}
        Using \cite[Problem 8, p.18]{billingsley1999}, this can be extended to the uniform convergence
        \begin{align*}
            \sup_{z\in\mathbb{C}^+} 
            \left| \E [ \su_{n,t} (z ) ] -  \su_t (z) \right| = o(1), \quad n\to\infty. 
        \end{align*}
        Let $\eta_n \to 0$ sufficiently slowly, e.g., $\eta_n = \mathcal{O}\lb n^{-m} \rb$ for some fixed $m\in\N.$
        This implies for $z_n = \lambda_{k,t} + i \eta_n$ 
        \begin{align} \label{f1}
            \E [ \su_{n,t} (z_n ) ] -  \su_t (z_n) = o(1), \quad n\to\infty. 
        \end{align} 
        Since $\lambda_{k,t} \notin \operatorname{supp} (\underline F^{y, H}),$ it follows that 
        \begin{align} \label{f2}
            \su_t (z_n) = \su_t(\lambda_{k,t}) + o(1), \quad n\to\infty. 
        \end{align}
        Next, we get from \cite[(9.7.8)]{bai2004} for some sufficiently small $\varepsilon>0$
        \begin{align*}
            \PR \lb \lambda_1 (\bfS_{2,t} ) \geq  \lambda_{k,t} + \varepsilon \rb = o\lb n ^{-l} \rb 
        \end{align*}
        for all $l\in\N.$
        Thus, we have 
        \begin{align}
             \E \left| \su_{n,t} (z_n) - \su_{n,t}(\lambda_{k,t})  \right|   
           & = \E \left[ \left| \su_{n,t} (z_n) - \su_{n,t}(\lambda_{k,t}) \right| 
            I \{ \lambda_1 (\bfS_{2,t} ) < \lambda_{k,t} + \varepsilon \} 
            \right] \nonumber \\ & \quad 
            + \E \left[  \left| \su_{n,t} (z_n) - \su_{n,t}(\lambda_{k,t}) \right|  I \{ \lambda_1 (\bfS_{2,t} ) \geq  \lambda_{k,t} + \varepsilon \} 
            \right] \nonumber 
            \\ & = o(1).  \label{f3}
        \end{align}
        A combination of  \eqref{f1}, \eqref{f2} and \eqref{f3} concludes the proof. 
    \end{proof}
    In the following, we may assume without loss of generality that the spectrum of $\bfS_{2,t}$ is contained in the interval $I_t$ with probability $1$ (recall \eqref{eq_support_lambda}).
 Thus, we have  
    \begin{align} \label{eq_bound_D}
       \| \D_t \|\leq \frac{1}{\textnormal{dist}(I_t, \lambda_{k,t}) }
       \leq \frac{1}{ \min ( \textnormal{dist}(I_{t_0}, \lambda_{k,t_0}) , \textnormal{dist}(I_{t_0}, \lambda_{k,1}) )  }
       \lesssim 1,
    \end{align}
    with probability $1$, where the bound on the right-hand side does not depend on $t$. Analogous estimates are available for similarly defined matrices like $\D_{i,t}, \D_{i,j,t}$. 
    Moreover, by distinguishing the cases $\lambda_{k,t} < ( 1 - \sqrt{y_{(t)}})^2$ and $\lambda_{k,t} > ( 1 + \sqrt{y_{(t)}})^2$,  we have the estimate 
    \begin{align} \label{eq_bound_beta}
        |\beta_{i,t}(\lambda_{k,t})| =  \frac{1}{|\lambda_{k,t}| \left|  1/\lambda_{k,t} + n\inv \eta_i^\top \lb \bfS_{2,t} / \lambda_{k,t} - \bfI \rb\inv \eta_i \right|  }
        \leq 1.
    \end{align} 
 
    This implies 
    \begin{align*}
        \E |  \beta_{1,t} - \E [\beta_{1,t}] |^2 
        \lesssim 
        \E \left| n\inv  \lb \eta_1^\top \D_1 \eta_1 - \tr \D_1 \rb \right|^2 
        + \E \left| n\inv \lb \tr \D_1 - \tr \E \D_1 \rb \right|^2 = o(1). 
    \end{align*}
    Here, we used \cite[Lemma B.26]{bai2004} for the first summand and a martingale decomposition for the second summand. Similarly to \cite[Lemma 7.1.3]{diss}, we can show
    \begin{align*}
        \lambda_{k,t} \su_{n,t} (\lambda_{k,t}) = - \frac{1}{\nt} \sum\limits_{j=1}^{\nt} \beta_{j,t}(\lambda_{k,t})
    \end{align*}
    for $\lambda \notin \lb I_t \cup \{ 0\}\rb $. 
    Recall that $\su_{n,t}(\lambda_{k,t}) = \su_{n,t}$ denotes the Stieltjes transform of the companion matrix $\underline{\bfS}_{2,t} = \bfX_{2,t}^\top \bfX_{2,t} \in \R^{\nt \times \nt}$.
    This implies 
    \begin{align} \label{eq_beta_stieltjes}
        \lambda_{k,t} \E [  \su_{n,t} (\lambda_{k,t}) ] = - \E [ \beta_{j,t}(\lambda_{k,t}) ]
        \textnormal{ and } 
        \E \lb \beta_{1,t} (\lambda_{k,t}) + \lambda_{k,t} \E [ \su_{n,t}(\lambda_{k,t}) ] \rb^2 = o(1). 
    \end{align}
    Then, Lemma \ref{lem_conv_stieltjes} guarantees that $\E [ \su_{n,t}(\lambda_{k,t})] $ appearing in \eqref{eq_beta_stieltjes} can be replaced by $\su_t (\lambda_{k,t})$.
  
\begin{proof}[Proof of Lemma \ref{lem_tr_S_D}]
By an application of the Sherman-Morrison formula, we obtain
      \begin{align} \label{eq_sher_mor}
        \D_t = \D_{i,t} - \frac{1}{n} \beta_{i,t} \D_{i,t} \eta_i \eta_i^\top \D_{i,t}.
    \end{align}
    Thus, we have
    \begin{align}
         \frac{1}{n} \tr \bfS_{2,s} \D_t
         = \frac{1}{n} \sum\limits_{i=1}^{\ns} \frac{1}{n} \eta_i^\top \D_t \eta_i 
         = \frac{1}{n} \sum\limits_{i=1}^{\ns} \left\{  \frac{1}{n} \eta_i^\top \D_{i,t} \eta_i 
         - \frac{1}{n} \beta_{i,t} \lb \eta_i^\top \D_{i,t} \eta_i \rb^2 \right\} \label{sum1}
    \end{align}
     Using \cite[Lemma B.26]{bai2004}, \eqref{eq_bound_D} and the fact that $\D_{i,t}$ is independent of $\eta_i$, we see that
    \begin{align*}
        \E \left| 
         \frac{1}{n^2} \sum\limits_{i=1}^{\ns} \lb \eta_i^\top \D_{i,t} \eta_i 
        - \tr \D_{i,t} 
         \rb 
        \right| 
         = o(1).
    \end{align*}
     Using \eqref{eq_beta_stieltjes} and \cite[Lemma B.26]{bai2004}, we may replace $\beta_{i,t}$ in the second summand of \eqref{sum1} by $-\lambda\su_{n,t}(\lambda)$.
     Proceeding similarly as for the first summand in \eqref{sum1}, we obtain
     \begin{align*}
         \frac{1}{n^3} \sum\limits_{i=1}^{\ns}   \lb \eta_i^\top \D_{it} \eta_i \rb^2
         = \frac{\ns}{ n}\lb \frac{1}{n} \tr \D_{t} \rb^2 
          + o_{\PR}(1). 
     \end{align*}
        Recall that 
   \begin{align} \label{conv_stieltjes}
       \frac{1}{n} \tr \D_t = y m_t + \op. 
   \end{align}
     Combining these estimates, we see that
     \begin{align*}
          \frac{1}{n} \tr \bfS_{2,s} \D_t
          = \frac{\ns}{n} \lb \frac{1}{n} \tr \D_t 
          + \lambda_{k,t} \su_t \lb\frac{1}{n} \tr \D_t \rb ^2 \rb + o_{\PR}(1)
          = s y m_t \lb 1 + \lambda_{k,t} y \su_t  m_t \rb + o_{\PR} (1). 
     \end{align*}
     Using \eqref{eq_m_mu} twice as well as \eqref{eq_stieltjes_MP}, we get
        \begin{align*}
            s m_t ( 1 + y \su_t m_t) 
            & =  s m_t \left\{  1 + y  m_t \lambda_{k,t} \lb - \frac{1 - y_{(t)} }{\lambda_{k,t}} + y_{(t)} m_t \rb  \right\} \\ 
            & = \frac{s}{t} m_t \left\{  t - \lambda_{k,t} y m_t - y m_t \lb t - y - \lambda_{k,t} - \lambda_{k,t} y m_t \rb \right\} \\ 
            & = \frac{s}{t} m_t \lb  t - \lambda_{k,t} y m_t - y  \rb 
             = \frac{s}{t} \lb 1 + \lambda_{k,t} m_t \rb,
        \end{align*}
     which concludes the proof of Lemma \ref{lem_tr_S_D}. 
\end{proof}

\begin{proof}[Proof of Lemma \ref{lem_tr_S_D_unif}]
 The equality $  1 + \lambda_{k,t} m_t = -  m_{1,t} $ follows by a direct calculation. Considering the proof of Lemma \ref{lem_tr_S_D} closely and recalling that the estimates \eqref{eq_bound_D} and \eqref{eq_bound_beta} do not depend on $t$, we see that  
    it remains to strengthen the convergence in \eqref{conv_stieltjes}. By using a martingale decomposition and \eqref{eq_bound_D}, \eqref{eq_bound_beta} and \cite[Lemma B.26]{bai2004}, we see that
    \begin{align}
       \sup_{t\in [t_0,1]} \E \lb \frac{1}{\sqrt{n}} \tr \lb \D_t - \E \D_t \rb  \rb^2 = o(1).
    \end{align}
    Using \cite[Theorem 4.5]{dornemann2021linear}, this implies that 
    \begin{align*}
         \frac{1}{n} \tr \D_t = y m_t + o_{\PR}\lb n^{-1/2} \rb,
    \end{align*}
    where the $o_{\PR}\lb n^{-1/2} \rb$-term is independent of $t$,
    and the proof of Lemma \ref{lem_tr_S_D_unif} concludes. 
\end{proof}

\begin{proof}[Proof of Lemma \ref{lem_tr_S_D_S_D}]
      Using \eqref{eq_sher_mor}, we decompose
      \begin{align*}
          \frac{1}{n}   \tr \bfS_{2,s} \D_s \bfS_{2,s}\D_t 
           = \frac{1}{n^3}\sum\limits_{i,j=1}^{\ns} 
           \eta_i^\top \D_s \eta_j \eta_j^\top \D_t \eta_i 
           = T_1 + T_2,
      \end{align*}
      where
      \begin{align*}
          T_1 & = T_1 (n,s,t) 
          = \frac{1}{n^3}\sum\limits_{i=1}^{\ns} 
           \eta_i^\top \D_s \eta_i \eta_i^\top \D_t \eta_i, \\ 
           T_2 & = T_2 (n,s,t)
           = \frac{1}{n^3}\sum_{\substack{i,j=1, \\ i\neq j }}^{\ns} 
           \eta_i^\top \D_s \eta_j \eta_j^\top \D_t \eta_i .
      \end{align*}
      We will analyze the terms $T_1$ and $T_2$ using the techniques presented in the proof of Lemma \ref{lem_tr_S_D}. 
      \paragraph*{Analysis of $T_1$}
        For $T_1$, we see that 
        \begin{align*}
            T_1 & = 
            \frac{1}{n} \sum\limits_{i=1}^{\ns}
            \Big( \frac{1}{n^2} \eta_i^\top \D_{i,s} \eta_i \eta_i^\top \D_{i,t} \eta_i
            - \frac{1}{n^3} \beta_{i,s} \eta_i^\top \D_{i,s} \eta_i \eta_i^\top \D_{i,s} \eta_i \eta_i^\top \D_{i,t} \eta_i
            \\ & \quad - \frac{1}{n^3} \beta_{i,t}  \eta_i^\top \D_{i,s} \eta_i \eta_i^\top \D_{i,t} \eta_i \eta_i^\top \D_{i,t} \eta_i
             + \frac{1}{n^4} \beta_{i,s} \beta_{i,t}  \eta_i^\top \D_{i,s}  \eta_i \eta_i^\top \D_{i,s} \eta_i \eta_i^\top \D_{i,t} \eta_i \eta_i^\top \D_{i,t} \eta_i
            \Big) \\ 
            & = \frac{\ns}{n}
            \Big\{ 
            n^{-2} \tr \D_{s} \tr \D_t 
            + n^{-3} \lambda_{k,s} \su_s \lb \tr \D_s \rb^2 \tr \D_t 
            + n^{-3} \lambda_{k,t} \su_t \tr \D_s \lb \tr \D_t \rb^2 
            \\ & \quad + n^{-4} \lambda_{k,s} \lambda_{k,t} \su_s \su_t \lb \tr \D_s \rb^2 \lb \tr \D_t \rb^2 
            \Big\} 
            + \op \\ 
            & = s y^2 m_s m_t 
            \lb 1 
            + y \lambda_{k,s} \su_s m_s 
            + y \lambda_{k,t} \su_t  m_t 
            + y^2 \lambda_{k,s} \lambda_{k,t} \su_s \su_t m_s m_t
            \rb + \op .
        \end{align*}
      \paragraph*{Analysis of $T_2$}
      We write
      \begin{align*}
          T_2 & = 
          \frac{1}{n}\sum_{\substack{i,j=1, \\ i\neq j }}^{\ns} \Big( 
           \frac{1}{n^2} \eta_i^\top \D_{i,s} \eta_j \eta_j^\top \D_{i,t} \eta_i
           - \frac{1}{n^3} \beta_{i,s}  \eta_i^\top \D_{i,s} \eta_i \eta_i^\top \D_{i,s} \eta_j \eta_j^\top \D_{i,t} \eta_i \\ & \quad
           - \frac{1}{n^3} \beta_{i,t}  \eta_i^\top \D_{i,s} \eta_j \eta_j^\top \D_{i,t} \eta_i \eta_i^\top \D_{i,t} \eta_i 
           + \frac{1}{n^4} \beta_{i,s} \beta_{i,t}  \eta_i^\top \D_{i,s} \eta_i \eta_i^\top \D_{i,s} \eta_j \eta_j^\top \D_{i,t} \eta_i \eta_i^\top \D_{i,t} \eta_i 
           \Big) \\ 
           & = s T_3  \lb 1 + \lambda_{k,s} y \su_s m_s
           + \lambda_{k,t} y \su_t m_t
           + \lambda_{k,s} \lambda_{k,t} y^2 \su_s \su_t m_s m_t \rb 
           + \op ,
      \end{align*}
      where
      \begin{align*}
          T_3 = T_3(n,s,t)
          = \frac{1}{n^2} \sum_{j=1}^{\ns} \eta_j^\top \D_s \D_t \eta_j .
      \end{align*}
      Thus, it remains to investigate the term $T_3$, which will be done in the following step.
      \paragraph*{Analysis of $T_3$}
      Again, we decompose
      \begin{align*}
          T_3 & = 
          \frac{1}{n} \sum_{j=1}^{\ns} \Big( \frac{1}{n}
          \eta_j^\top \D_{j,s} \D_{j,t} \eta_j
          - \frac{1}{n^2} \beta_{j,s} \eta_j^\top \D_{j,s} \eta_j \eta_j^\top \D_{j,s} \D_{j,t} \eta_j
          - \frac{1}{n^2} \beta_{j,t} \eta_j^\top \D_{j,s} \D_{j,t} \eta_j \eta_j^\top \D_{j,t} \eta_j
          \\ & + \frac{1}{n^3} \beta_{j,s} \beta_{j,t} \eta_j^\top \D_{j,s} \eta_j \eta_j^\top \D_{j,s} \D_{j,t} \eta_j \eta_j^\top \D_{j,t} \eta_j
          \Big) \\ 
          & = \frac{\ns}{n}
          \Big( n\inv \tr \D_s \D_t 
          + \lambda_{k,s} \su_s n^{-2} \tr \D_s \tr \D_s \D_t
          + \lambda_{k,t} \su_t n^{-2} \tr \D_t \tr \D_s \D_t \\ & \quad 
          + \lambda_{k,s} \lambda_{k,t} \su_s \su_t n^{-3} \tr \D_s \tr \D_t \tr \D_s \D_t
          \Big) 
          + \op \\
          = & s n\inv \tr \D_s \D_t 
          \Big( 1 + \lambda_{k,s} y \su_s m_s 
          + \lambda_{k,t} y \su_t m_t 
          + \lambda_{k,s} \lambda_{k,t} y^2 \su_s \su_t m_s m_t 
          \Big) + \op .
      \end{align*}
      The term $n\inv \tr \D_s \D_t $ is analyzed in Lemma \ref{lem_conv_tr_DsDt}. 
      \paragraph*{Conclusion}
      Putting all three steps together, we obtain
      \begin{align*}
        &   \frac{1}{n} \tr \bfS_{2,s} \D_s \bfS_{2,s} \D_t  
         = T_ 1 + T_2 
        =  s y^2 m_s m_t 
            \lb 1 
            + y \lambda_{k,s} \su_s m_s 
            + y \lambda_{k,t} \su_t  m_t 
            + y^2 \lambda_{k,s} \lambda_{k,t} \su_s \su_t m_s m_t
            \rb \\ & \quad 
            + s T_3  \lb 1 + \lambda_{k,s} y \su_s m_s
           + \lambda_{k,t} y \su_t m_t
           + \lambda_{k,s} \lambda_{k,t} y^2 \su_s \su_t m_s m_t \rb \\
           & =   s y^2 m_s m_t 
            \lb 1 
            + y \lambda_{k,s} \su_s m_s 
            + y \lambda_{k,t} \su_t  m_t 
            + y^2 \lambda_{k,s} \lambda_{k,t} \su_s \su_t m_s m_t \rb  \\ & \quad 
            +  s^2 n\inv \tr \D_s \D_t 
          \Big( 1 + \lambda_{k,s} y \su_s m_s 
          + \lambda_{k,t} y \su_t m_t 
          + \lambda_{k,s} \lambda_{k,t} y^2 \su_s \su_t m_s m_t 
          \Big)^2
            + \op \\ 
            & = 
            s \left\{ y^2 m_s m_t + 
            s n\inv \tr \D_s \D_t 
            \lb 1 
            + y \lambda_{k,s} \su_s m_s 
            + y \lambda_{k,t} \su_t  m_t 
            + y^2 \lambda_{k,s} \lambda_{k,t} \su_s \su_t m_s m_t
            \rb
            \right\} \\ & \quad \times  
             \lb 1 
            + y \lambda_{k,s} \su_s m_s 
            + y \lambda_{k,t} \su_t  m_t 
            + y^2 \lambda_{k,s} \lambda_{k,t} \su_s \su_t m_s m_t
            \rb + \op 
            \\ 
            & = 
            s \left\{ y^2 m_s m_t + 
             \frac{s y m_s m_t}{1 - s y \lambda_{k,t} \lambda_{k,s} \su_s \su_t m_s m_t }
            \lb 1 
            + y \lambda_{k,s} \su_s m_s 
            + y \lambda_{k,t} \su_t  m_t 
            + y^2 \lambda_{k,s} \lambda_{k,t} \su_s \su_t m_s m_t
            \rb
            \right\} \\ & \quad \times  
             \lb 1 
            + y \lambda_{k,s} \su_s m_s 
            + y \lambda_{k,t} \su_t  m_t 
            + y^2 \lambda_{k,s} \lambda_{k,t} \su_s \su_t m_s m_t
            \rb + \op .
      \end{align*}
\end{proof}

\begin{proof}[Proof of Lemma \ref{lem_conv_tr_DsDt}]
   We will make use of a decomposition of $\D_t$ derived in \cite{dornemann2021linear, diss}. However, we emphasize that we cannot directly employ the results of this work, where the quantity $\tr \lb   \E_j[ \D_{j,s} ] \E_j[ \D_{j,t} ] \rb $ has been studied for $1\leq j \leq \ns $. (Here, $\E_j$ denotes the expectation conditionally on $\eta_1, \ldots, \eta_j$.) In contrast, when omitting the conditional expectation, the term $\tr  \D_s \D_t$ admits a subtle different structure and needs to be studied carefully. 
    Similarly to \cite[(6.2.5)]{diss}, we have the following decomposition for $\D_s$,
    \begin{align} 
		\D_{s} = & - 
    \frac{1}{\lambda_{k,s} - \frac{\ns}{n} b_s} \bfI
  +  \frac{1}{\lambda_{k,s} - \frac{\ns}{n} b_s}
  \sum\limits_{ 1 \leq i \leq \ns} \beta_{i,s}  
  \eta_{i} \eta_{i}^\top \D_{i,s}
  - \frac{\ns}{n}  \frac{b_s}{\lambda_{k,s} - \frac{\ns}{n} b_s}
  \D_{s} \nonumber \\
		= & -  \frac{1}{\lambda_{k,s} - \frac{\ns}{n} b_s} \bfI 
		+ b_{s} \mathbf{A}_s + \mathbf{B}_s + \mathbf{C}_s,\label{decomp_D}
	\end{align}
	where
	\begin{align*}
	 	\mathbf{A}_s &= \mathbf{A}_s(\lambda_{k,s}) 
   =   \frac{1}{\lambda_{k,s} - \frac{\ns}{n} b_s} 
   \sum\limits_{\substack{i=1}}^{\ns}  
   \lb n\inv \eta_{i} \eta_{i}^\top - n\inv \T \rb \D_{i,s} , 
   \\
	 	\mathbf{B}_s &= \mathbf{B}_s(\lambda_{k,s}) 
   =  \frac{1}{\lambda_{k,s} - \frac{\ns}{n} b_s}
   \sum\limits_{\substack{i=1  }}^{\ns} \lb \beta_{i,s}  - b_{s} \rb 
	 	n\inv \eta_{i} \eta_{i}^\top \D_{i,s} , 
   \\
	 	\mathbf{C}_s &= \mathbf{C}_s(\lambda_{k,s}) 
   =   \frac{b_s}{\lambda_{k,s} - \frac{\ns}{n} b_s} 
   n \inv \sum\limits_{\substack{i=1  }}^{\ns} 
	 	\lb \D_{i,s}  - \D_{s}  \rb. 
	\end{align*}
 Our goal is to apply this decomposition to $( 1/n) \tr \D_s  \D_t $ and to identify the contributing terms. For this purpose, let
 \begin{align*}
     b_t(\lambda_{k,t}) & = b_t = \frac{1}{1 + n\inv \E [ \tr \D_t ]}.
 \end{align*} 
 Similarly to \cite{diss}, we see that terms involving $\bfB_s$ and $\bfC_s$ are asymptotically negligible, among others. 
 More precisely, we have
 \begin{align} 
    &  n\inv  \tr \lb \D_s  \D_t  \rb \nonumber \\
    & = n\inv b_s \tr \lb  \bfA_s  \D_t  \rb 
     + \frac{p}{n} \frac{1}{\lb \lambda_{k,s} - \frac{\ns}{n} b_s \rb \lb \lambda_{k,t} - \frac{\nt}{n} b_t \rb }
     + o_{\PR}(1), \label{a1}
 \end{align}
 and
 \begin{align} \label{a2}
      n\inv  \tr \lb \bfA_s  \D_t  \rb 
      = A_{1,s,t}  + o_{\PR}(1),
 \end{align}
 where
 \begin{align*}
     A_{1,s,t} & = - \frac{1}{n^3 \lb \lambda_{k,s} - \frac{\ns}{n} b_s \rb }
      \sum\limits_{1 \leq i \leq \ns } \beta_{i,t} \eta_{i}^\top \D_{i,s}   \D_{i,t}
		\eta_{i} \eta_{i}^\top \D_{i,t} 
   \eta_{i}.
 \end{align*}
    Using similar techniques as in the proof of Lemma \ref{lem_tr_S_D}, we get
    \begin{align} 
         A_{1,s,t} 
       &  = \frac{\ns}{n^3}  \frac{\lambda_{k,t}}{\lambda_{k,s} - \frac{\ns}{n} b_s}  \E [ \su_{n,t}] \tr \lb \D_{s}  \D_{t}  \rb 
        \tr  \D_{t} 
        + o_{\PR}(1) \nonumber \\ 
        &  = \frac{\ns p}{n^3} \lambda_{k,t} \E [ \su_{n,t}] \tr \lb \D_{s}  \D_{t}  \rb 
         \frac{1}{\lb \lambda_{k,s} - \frac{\ns}{n} b_s \rb \lb \lambda_{k,t} - \frac{\nt}{n} b_t \rb  }
        + o_{\PR}(1), \label{a3}
    \end{align}
    where we used the decomposition of $\D_{t}$ in \eqref{decomp_D} for the second equality sign. 
    Combining \eqref{a1}, \eqref{a2} and \eqref{a3} gives 
    \begin{align*}
        & \frac{1}{n} \tr \lb \D_{s}  \D_{t}  \rb 
        \lb 1 - \frac{\ns p }{n^2} \lambda_{k,s} \lambda_{k,t} \E [ \su_{n,s}] \E [ \su_{n,t}]
       \frac{1}{\lb \lambda_{k,s} - \frac{\ns}{n} b_s \rb \lb \lambda_{k,t} - \frac{\nt}{n} b_t \rb  }
       \rb  \\ 
        & =  \frac{p}{n}\frac{1}{\lb \lambda_{k,s} - \frac{\ns}{n} b_s \rb \lb \lambda_{k,t} - \frac{\nt}{n} b_t \rb  } 
        + o_{\PR} (1),
    \end{align*}
    and, consequently, 
            \begin{align*}
        & \frac{1}{n} \tr \lb \D_{s}  \D_{t}  \rb 
        \lb 1 - \frac{\ns p }{n^2}  \E [ \su_{n,s}] \E [ \su_{n,t}]
       \frac{1}{\lb 1 + \frac{\ns}{n} \E [ \su_{n,s}] \rb \lb 1 + \frac{\nt}{n} \E [ \su_{n,t}] \rb  }
       \rb  \\ 
        & =  \frac{p}{n} \frac{1}{\lambda_{k,s} \lambda_{k,t}} \frac{1}{\lb 1 + \frac{\ns}{n} \E [ \su_{n,s}] \rb \lb 1 + \frac{\nt}{n} \E [ \su_{n,t}] \rb  } 
        + o_{\PR} (1).
    \end{align*}
    We define 
    \begin{align*}
        a_n( s,t) 
  & = \frac{\ns p }{n^2}  \E [ \su_{n,s}] \E [ \su_{n,t}]
       \frac{1}{\lb 1 + \frac{\ns}{n} \E [ \su_{n,s}] \rb \lb 1 + \frac{\nt}{n} \E [ \su_{n,t}] \rb  }.
    \end{align*}
    Similarly to \cite[Lemma 7.1.7]{diss}, we have $|a_n( s, t)|<1$. 
        Using $\E [\su_{n,s}] \to \su_s$ for $n\to\infty$ and \eqref{eq_m_mu},
  it holds that
    \begin{align} \label{eq_def_a}
        a_n ( s, t) \to 
        s y \frac{\su_s \su_t}{\lb 1 + s \su_s \rb \lb 1+ t \su_t\rb }
        = s y \lambda_{k,s} \lambda_{k,t} \su_s \su_t m_s m_t,
    \end{align}
 which implies \eqref{eq_conv_tr_DsDt}. 

\end{proof}

\end{appendix}
 
\textbf{Acknowledgements.} 
The work of Nina Dörnemann  was partially supported by the  
 DFG Research unit 5381 {\it Mathematical Statistics in the Information Age}, project number 460867398. 

   \setlength{\bibsep}{1pt}
\begin{small}
\bibliography{references}
\end{small}

\end{document}